\documentclass[11pt]{amsart}
\usepackage[letterpaper,margin=1in]{geometry}                
\usepackage{amssymb,amsmath,amsfonts,stmaryrd}
\usepackage{xcolor}
\usepackage[colorlinks,
             linkcolor=blue,
             citecolor=green!70!black,
             pdfproducer={pdfLaTeX},
             pdfpagemode=None,
             bookmarksopen=true
             bookmarksnumbered=true]{hyperref}
\usepackage{graphicx}
\usepackage{multicol}

\usepackage{tikz}
\usetikzlibrary{decorations.markings,arrows,decorations.pathreplacing}
\usepackage{arydshln}

\allowdisplaybreaks

\newcommand\arXiv[1]{\href{http:\sslash arxiv.org/abs/#1}{\nolinkurl{arXiv:#1}}}
\newcommand\MRnumber[1]{\href{http:\sslash www.ams.org/mathscinet-getitem?mr=#1}{\nolinkurl{MR#1}}}
\newcommand\DOI[1]{\href{http:\sslash dx.doi.org/#1}{\nolinkurl{DOI:#1}}}
\newcommand\MAILTO[1]{\href{mailto:#1}{\nolinkurl{#1}}}

\newtheorem{theorem}{Theorem}
\newtheorem{lemma}[subsubsection]{Lemma}
\newtheorem{proposition}[subsubsection]{Proposition}
\newtheorem{conjecture}{Conjecture}
\newtheorem{corollary}[subsubsection]{Corollary}

\theoremstyle{definition}
\newtheorem{example}[subsubsection]{Example}

\usepackage{dsfont}
\renewcommand\mathbb\mathds

\newcommand\bC{\mathbb C}

\newcommand\bF{\mathbb F}

\newcommand\bM{\mathbb M}

\newcommand\bR{\mathbb R}

\newcommand\bZ{\mathbb Z}

\newcommand\cC{\mathcal C}

\newcommand\cM{\mathcal M}
\newcommand\cN{\mathcal N}

\newcommand\cZ{\mathcal Z}

\newcommand\rA{\mathrm A}
\newcommand\rB{\mathrm B}
\newcommand\rC{\mathrm C}
\newcommand\rD{\mathrm D}

\newcommand\rF{\mathrm F}

\newcommand\rJ{\mathrm J}

\newcommand\rO{\mathrm O}

\newcommand\rU{\mathrm U}

\usepackage[mathscr]{euscript}

\DeclareMathOperator\homology{H}
\renewcommand\H{\homology}

\DeclareMathOperator\cycles{Z}
\newcommand\Z{\cycles}

\DeclareMathOperator\chains{C}
\newcommand\C{\chains}

\renewcommand\d{\mathrm d}

\newcommand\longto\longrightarrow
\newcommand\mono\hookrightarrow
\newcommand\epi\twoheadrightarrow

\newcommand\<\langle
\renewcommand\>\rangle
\newcommand\sminus\smallsetminus

\newcommand\Co{\mathrm{Co}}

\newcommand\Suz{\mathrm{Suz}}

\newcommand\SU{\mathrm{SU}}

\DeclareMathOperator\Aut{Aut}
\DeclareMathOperator\Sym{Sym}
\DeclareMathOperator\Alt{Alt}

\DeclareMathOperator\Fer{Fer}

\DeclareMathOperator\Sq{Sq}

\newcommand\Vect{\cat{Vect}}

\newcommand\tr{\mathrm{tr}}

\newcommand\define[1]{\emph{#1}}
\newcommand\cat[1]{\textsc{#1}}

\title{The Moonshine Anomaly}
\author{Theo Johnson-Freyd}
\date{\today}
\email{\MAILTO{theojf@pitp.ca}}
\address{Perimeter Institute for Theoretical Physics, Waterloo, Ontario}

\begin{document}
\begin{abstract}
  The anomaly for the Monster group $\bM$ acting on its natural (aka moonshine) representation $V^\natural$ is a particular cohomology class $\omega^\natural \in \H^3(\bM,\rU(1))$ that arises as a conformal field theoretic generalization of the second Chern class of a representation.  
  This paper shows that $\omega^\natural$ has order exactly $24$ and is not a Chern class.  In order to perform this computation, this paper introduces a finite-group version of T-duality, which is used to relate $\omega^\natural$ to the anomaly for the Leech lattice~CFT.
\end{abstract}
\maketitle


\section{Introduction}

Whenever a finite group $G$ acts on a holomorphic conformal field theory $V$, there is a corresponding \define{anomaly} $\omega_V \in \H^3(G,\rU(1))$ that measures the obstruction to gauging (also called orbifolding) the $G$-action. We will review the construction of $\omega_V$ in Section~\ref{orbifolds}: Section~\ref{physical picture} describes the physical picture; Section~\ref{why conformal nets} discusses mathematical foundations; Section~\ref{t-duality} provides a method to  calculate anomalies; and Section~\ref{super case} addresses the fermionic version.
Anomalies can be thought of as a type of ``characteristic class'' for the action of a group not on a vector space but on a conformal field theory, 
and have played an increasingly important role in recent work on moonshine-type phenomena \cite{MR2500561,MR2681787,MR3108775,CLW2016}.
 The most famous example of a finite group acting on a conformal field theory is certainly the Monster group $\bM$ acting on its natural ``moonshine'' representation $V^\natural$. The main result of this paper is a calculation of the corresponding \define{moonshine anomaly} $\omega^\natural \in \H^3(\bM,\rU(1))$, showing in particular that it does not vanish:

\begin{theorem} \label{moonshine theorem}
  The order of $\omega^\natural \in \H^3(\bM,\rU(1))$ is exactly $24$. Although Chern and fractional Pontryagin classes can arise as anomalies, $\omega^\natural$ is neither a Chern nor a fractional Pontryagin class of any representation of $\bM$.
\end{theorem}

The proof of Theorem~\ref{moonshine theorem} occupies most of Section~\ref{M}. Specifically, in Section~\ref{large primes} we check that $\H^3(\bM,\rU(1))$ has no elements with order divisible by the primes $p=11$ or $p\geq 17$; Section~\ref{small primes} shows that the order of $\omega^\natural$ is not divisible by the primes $p=5$, $7$, or $13$, and is divisible by $3$ but not $9$; and we show that $8$ but not $16$ divides the order of $\omega^\natural$ in Section~\ref{monster 2}. 
Section~\ref{section not c2}  proves that $\omega^\natural$ is not a Chern class. Section~\ref{closing remarks} contains some further remarks about the value of $\omega^\natural$.

A main step in the proof is the following result due jointly to the author and D.\ Treumann. Let $\Co_0 = 2.\Co_1$ denote Conway's largest group, and $\Lambda$ the Leech lattice.

\begin{theorem}[\cite{JFT}] \label{conway theorem}
  $\H^3(\Co_0,\rU(1))$ is cyclic of order $24$, generated by $\frac{p_1}2(\Lambda \otimes \bR)$, the first fractional Pontryagin class  of the $24$-dimensional defining representation of $\Co_0$. \qed
\end{theorem}

By Example~\ref{example permutation}, the anomalies of permutation orbifolds are aways Pontryagin classes.
As explained in Example~\ref{example fer 24}, the class $\frac{p_1}2(\Lambda \otimes \bR) \in \H^3(\Co_0,\rU(1))$ is  the anomaly for the action of $\Co_0$ on the ``super moonshine'' module $\mathrm{Fer}(24)\sslash \bZ_2$ of \cite{MR2352133,MR3376736}, and  Pontryagin and  Chern classes of finite-dimensional representations of a group $G$ can always be realized as the anomalies of actions of $G$ on free fermion models.
  There is an a priori upper bound on the orders of Pontryagin and Chern classes of representations that are (like $\Lambda$) defined over~$\bZ$, coming from the fact that the fourth integral group cohomology of $\mathrm{GL}(N,\bZ)$ for $N \gg 0$ is $\bZ_{24} \oplus \bZ_2^2$~\cite{MR753419}.
Generalizing from characteristic classes to CFT anomalies suggests:
\begin{conjecture}
  Let $V$ be a holomorphic conformal field theory and $G \subset \Aut(V)$ a finite group of automorphisms such that both $V$ and the action of $G$ are defined over $\bZ$. Then the corresponding anomaly $\omega_V \in \H^3(G,\rU(1))$ has order dividing $24$.
\end{conjecture}
We will not try to make precise in this paper what it should mean to say that a conformal field theory is ``defined over $\bZ$.'' Almost surely the correct notion is that of a vertex operator algebra equipped with an \define{integral form} as in \cite{MR2928458}. However, for technical reasons we do not model conformal field theories in this paper as vertex operator algebras, using instead conformal nets of von Neumann algebras (see Section~\ref{why conformal nets}), and we do not claim to know what a ``conformal net over~$\bZ$'' should~be.

As discussed in \cite{MO69222,MR3539377}, one can predict the order of the moonshine anomaly $\omega^\natural$ using vertex operator algebraic methods familiar to moonshine theorists. 
  We do not follow that route: the calculations in this paper depend instead on methods native to the cohomology of finite groups (e.g.\ maximal subgroups, spectral sequences) and to topological field theory (e.g.\ fusion categories). The inputs we use from conformal field theory are the existence of $V^\natural$, and hence of $\omega^\natural$, and its relationship to the Leech lattice conformal field theory~$V_\Lambda$. Specifically, there is a well-known~\cite{MR554399} agreement
   between centralizers of certain elements of $\bM$ and of certain elements of $\Aut(V_\Lambda) = \hom(\Lambda,\rU(1)).\Co_0$, which we will understand as an example of \define{finite group T-duality} introduced in Section~\ref{t-duality}.
T-duality allows information about anomalies to be moved between conformal field theories related by a cyclic orbifold. We use it extensively in the proof of Theorem~\ref{moonshine theorem}.

Our proof of Theorem~\ref{moonshine theorem} avoids computing much about the group $\H^3(\bM,\rU(1))$. Nevertheless, it is very tempting to speculate the following analog of Theorem~\ref{conway theorem}:

\begin{conjecture} \label{H4M conjecture}
  $\H^3(\bM,\rU(1)) \cong \bZ_{24}$.
\end{conjecture}

We address this conjecture in Section~\ref{closing remarks} without providing much evidence to support it. In Section~\ref{section not c2}, we show that Conjecture~\ref{H4M conjecture} implies:

\begin{conjecture} \label{c2=0 conjecture}
  For every complex representation $V$ of $\bM$, $$c_2(V) = 0 \in \H^3(\bM,\rU(1)).$$
\end{conjecture}

\subsection{Notation and conventions}

We mostly follow the ATLAS \cite{ATLAS} for notation for finite groups. For example, when referring to a group, ``$n$'' denotes the cyclic group of order $n$. We also call this group $\bZ_n$, which is its standard name in physics texts, and $\bF_n$ when $n$ is prime and we are thinking of it as a field. (Mathematicians sometimes use $\bZ_n$ to denote instead the ring of $n$-adic integers.) Elementary abelian groups are denoted $n^k$ and extraspecial groups $n^{1+k}$. An extension with normal subgroup $N$ and cokernel $G$ is denoted $N.G$ or occasionally $NG$; an extension which is known to split is written $N:G$. The conjugacy classes of elements of order $n$ in a group $G$ are named $n\rA$, $n\rB$, \dots, ordered by increasing size of the class (decreasing size of the centralizer).

The \define{exponent} of a finite group $G$ is the smallest $n$ such that $g^n = 1$ for all $g\in G$.
If $G$ is a finite group and $A$ a module thereof, we write $\H^\bullet(G,A)$ for the group cohomology of $G$ with coefficients in $A$; this should cause no confusion, since the cohomology of $G$-as-a-space is trivial.
When $G$ is a Lie group, we write $\H^\bullet(\rB G)$ to emphasize that the cohomology we consider depends just on the classifying space $\rB G$ of $G$.
We primarily use $\rU(1)$-coefficients instead of the more mathematically-common $\bZ$-coefficients because the former is more physically meaningful. Of course for a finite group $\H^\bullet(G,\rU(1)) \cong \H^{\bullet+1}(G,\bZ)$ is finite of exponent dividing the order of $G$. 

The \define{Pontryagin dual} of a possibly-infinite abelian group $A$ is $\widehat{A} = \hom(A,\rU(1))$. We will extend the notation $\widehat{(-)}$ in Section~\ref{t-duality}; the meaning will be clear from context.

\subsection{Acknowledgments}

I discussed substantial portions of this work with David Treumann, who pointed out that $\omega^\natural$ might not be a Chern class, and with Marcel Bischoff, who told me a separate conformal-net-theoretic proof of Example~\ref{example permutation} using the main results of  \cite{MR1410566,MR2100058}. I thank both of them for their time and attention.
 Leonard Soicher kindly shared with me matrices satisfying his presentation for $\Co_1$; these matrices 
 were first computed by Richard Parker and
 are reproduced in Appendix~\ref{Soicher matrices}.
 I would also like to thank Lakshya Bhardwaj, Richard Borcherds, Kevin Costello, Davide Gaiotto, Nora Ganter, Andre Henriques, David Jordan, and Alex Weekes for helpful suggestions and discussions, and the referees for their valuable comments. Research at the Perimeter Institute for Theoretical Physics is supported by the Government of Canada through the Department of Innovation, Science and Economic Development Canada and by the Province of Ontario through the Ministry of Research, Innovation and Science.

\section{Orbifolds of conformal field theories} \label{orbifolds}

In this section, we first review the basic theory of orbifolds of conformal field theories and the corresponding anomalies, beginning with the physical description in Section~\ref{physical picture} and then mentioning the appropriate mathematical details in Section~\ref{why conformal nets}. 
We then introduce a ``T-duality'' for finite groups in Section~\ref{t-duality}. In Section~\ref{super case} we discuss anomalies and T-duality when fermions are present.

\subsection{Physical picture} \label{physical picture}

Before focusing on a precise mathematical model, it is worth recalling the physical meaning of (``gauge'' aka ``orbifold'' aka ``'t Hooft'') anomalies. Much of this story originated in \cite{MR1003430,MR1128130} and is well known to experts; I learned the picture of orbifolds as a sandwich construction from K.\ Costello and D.\ Gaiotto, who said that it is not due to them but rather ``in the air'' and who were unable to provide sharper references.
It is a special case of the construction detailed in the series \cite{MR1940282,MR2026879,MR2076134,MR2137114,MR2259258}.

Suppose that one has a $d$-spacetime-dimensional quantum field theory $A$ and a finite group~$G$ acting on~$A$ by global symmetries. (We will care about conformal field theory in dimension $d=2$.) One can ask to ``gauge'' the $G$-symmetry --- to integrate over the choice of $G$ gauge field. This requires choosing a measure of integration for the gauge field. If $A$ were absent, one could choose a Dijkgraaf--Witten type measure as in \cite{MR1048699,MR1240583}, parameterized by a $\rU(1)$-cohomology class $[\alpha] \in \H^d(G,\rU(1))$. 

Why are Dijkgraaf--Witten actions parameterized by cohomology classes? Any degree-$d$ cochain $\alpha \in \rC^d(G,\rU(1))$ provides a putative measure of integration, but consistency under gauge symmetry forces $\alpha$ to be a cocycle, and cohomologous choices of $\alpha$ determine equivalent gauged theories. 
Actually, as emphasized in \cite{MR1240583}, the cohomology class $[\alpha]$ determines the QFT up to isomorphism, but not up to canonical isomorphism: to give an isomorphism between the $\alpha$ and $\alpha'$ theories requires providing a solution to $\d\beta = \alpha - \alpha'$, modulo varying $\beta \leadsto \beta + \d\gamma$.
We will therefore adopt the following mildly abusive notation. We will speak just about the class $[\alpha] \in \H^d(G,\rU(1))$, since that class determines the isomorphism type of the Dijkgraaf--Witten theory, but we will understand that extra data is needed in order to give a trivialization of a trivializable class $[\alpha]$. This extra data lives in a torsor for $\H^{d-1}(G,\rU(1))$, which vanishes in many of the cases where we will actually need a trivialization of $[\alpha]$.

(There is one further correction necessary, which will not matter for this paper. Namely, the most  general pure gauge theories are not parameterized by ordinary cohomology but rather by some generalized cohomology theory. For example, in a spin Dijkgraaf--Witten theory $\alpha$ can fail to be an ordinary cocycle if that failure is accommodated by a term that couples to the spin structure.) 

In the presence of the quantum field theory $A$, however, the same logic does not force $\d\alpha = 0$. Rather, it forces $\d\alpha = \omega$, where $\omega$ is some degree-$(d+1)$ cocycle determined by the way that $G$ acts on $A$
(as well as some auxiliary choices of ``coordinates'' for the action of $G$ on $A$; under changes of these ``coordinates,'' $\omega$ changes by the derivative of an explicit $d$-cochain).
 In particular, if $\omega$ represents a nontrivial class in cohomology, then the $G$-symmetry cannot be gauged. This cohomology class is called the  \define{anomaly} of the action of $G$ on $A$. Given a choice of trivialization of the anomaly, the resulting gauged theory will be denoted $A \sslash G$. Gauging is also called ``orbifolding'' for the following reason. Suppose $A$ quantizes a nonlinear sigma model with target $X$, and the $G$-action on $A$ comes from an action of $G$ on $X$. Then the gauged theory $A\sslash G$ quantizes a sigma model with target the orbifold $X\sslash G$.

There is a geometric way of thinking about anomalies that will be useful and has the benefit of clarifying exactly where anomalies live. To motivate it, recall the notion of ``chiral halves'' of a conformal field theory. Any quantum field theory, for example a conformal field theory, has an ``algebra'' of local observables, where ``algebra'' is in quotes because different axiomatizations (let alone different types of quantum field theories) involve different types of algebraic operations. In general quantum field theories are not determined by their algebras of local observables, but two-dimensional quantum field theories usually are. Suppose that $A$ is a two-dimensional conformal field theory. Its algebra of local observables, which we will in an abuse of notation also denote by $A$, has two subalgebras $V$ and $W$ consisting respectively of the \define{chiral} and \define{antichiral} observables. The theory $A$ is \define{heterotic} when $V \neq W$, and \define{holomorphic} when $W$ is trivial.

Just as every harmonic function is a sum of holomorphic and antiholomorphic pieces, so too is $A$ generated by its subalgebras $V$ and $W$, but just as there's a constant ambiguity in the choice of splitting of a function, so too are there relations between $V$ and $W$ in $A$.  To describe these relations involves moving briefly into three dimensions.

The algebras $V$ and $W$ are each algebras of local observables, but they do not typically define full quantum field theories the way $A$ does. Instead, each of $V$ and $W$ describes a (chiral or antichiral) boundary condition for a three-dimensional ``bulk'' quantum field theory, canonically determined by the boundary chiral algebra, and $A$ is the result of sandwiching a thin piece of the bulk QFT between the boundary conditions $V$ and $W$:
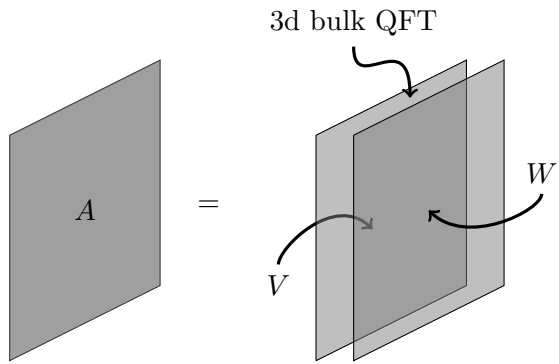
$$
\begin{tikzpicture}[baseline=(A)]
  \path[draw=black,fill=gray,opacity=.75] (0,0) -- (2,1) -- (2,4) -- (0,3) -- cycle;
  \path (1,2) node {$A$};
  \path (0,2) coordinate (A);
\end{tikzpicture}
\quad = \quad
\begin{tikzpicture}[baseline=(A)]
  \draw[->,very thick] (-.5,1) node (V) {$V$} (V) .. controls +(0,.5) and +(-.5,.5) .. (.75,1.75);
  \path[fill=gray,opacity=.5] (0,0) --  (2,1) -- (2,4) -- (0,3) -- (0,0);
  \path[draw=black] (0,0) --  (2,1) -- (2,4) -- (0,3) --  (0,0);
  \draw[->,very thick] (.5,4.5) node (Z) {3d bulk QFT} (Z) .. controls +(0,-1) and +(0,1) .. (1.25,3.5);
  \path[fill=gray,opacity=.5] (.5,0) -- (2.5,1) -- (2.5,4) -- (.5,3) -- cycle;
  \path[draw=black] (.5,0) -- (2.5,1) --  (2.5,4) -- (.5,3) -- cycle;
  \draw[->,very thick] (3,2.5) node (W) {$W$} (W) .. controls +(0,-.5) and +(.5,-.5) .. (1.5,2);
  \path (0,2) coordinate (A);
\end{tikzpicture}
$$

The best situation is when $V$ is ``rational'' (the name comes from the case of a string moving in $\bR/2\pi R\bZ$: the chiral algebra for such a theory is rational iff $R$ is), in which case the three-dimensional bulk theory is a topological quantum field theory of Reshetikhin--Turaev type with modular tensor category of line operators equal to the modular tensor category of vertex modules for $V$.  Here and throughout this section, we will write $\cZ(V)$ for the bulk quantum field theory for which $V$ describes a boundary condition. 
Reshetikhin--Turaev type TQFTs are determined by their MTCs of line operators, and so, provided $V$ is rational, we can write $\cZ(V)$ both for the 3d TQFT admitting $V$ as a boundary condition and also for the MTC of vertex modules for $V$.

To justify the letter ``$\cZ$,'' consider the case when $V$ is not chiral but rather topological. Then $V$ is fully determined by its fusion category of line defects, and so we will write ``$V$'' also for that fusion category. In this case, 
the three-dimensional bulk theory is the Turaev--Viro type TQFT determined by the fusion category $V$.
Turaev--Viro type TQFTs are a special case of Reshetikhin--Turaev type TQFTs; the MTC defining the Reshetin--Turaev theory is the Drinfel'd center of $V$, usually denoted $\cZ(V)$, hence the more general notation above \cite[Theorem 17.1]{MR3674995} (see also \cite{TuraevVirelizier,MR3093932}).
Indeed, a Turaev--Viro type TQFT depends only on the Morita equivalence class of its input fusion category $V$, which in turn can be reconstructed from the Drinfel'd center $\cZ(V)$ by \cite[Theorem 3.1]{MR2735754}, so the notation is fully justified in this case. (Another justifying example occurs one dimensional lower. Suppose $V$ is a separable symmetric Frobenius algebra. Then one can build a two-dimensional TQFT from $V$ determined by the condition that it admits a boundary condition with boundary observables in the algebra $V$; the bulk theory is described by the center of the algebra $V$.)

Combining the chiral and topological cases, one way to construct holomorphic two-dimensional conformal field theories is to find a chiral algebra $V$ and a fusion category $W$ and an identification $\cZ(V) \cong \cZ(W)$ (of 3d TQFTs or equivalently of modular tensor categories).
This is how we will  construct the orbifold  $A\sslash G$, where $A$ is a holomorphic conformal field theory with a $G$-action. The idea is to arrange a sandwich where the bulk theory is a pure $G$-gauge theory, and where the boundary condition ``$V$'' looks like $A$ together with ``Neumann'' boundary conditions for the bulk gauge fields. The boundary values of the gauge fields should of course couple to the $A$-system following the data of the action of $G$ on $A$. The boundary local observables for such a bulk-boundary system are precisely the $G$-fixed points in the algebra $A$, and so we will write $V = A^G$. This notation is particularly justified when the $G$ action on $A$ is faithful, as then this boundary condition is determined by its algebra of local observables, and the bulk theory $\cZ(A^G)$ is determined by the boundary chiral algebra $A^G$. Since $\cZ(A^G)$ is a pure gauge theory in three dimensions, it is a Dijkgraaf--Witten theory for some class $\omega \in \H^3(G,\rU(1))$, which we will see is precisely the anomaly described earlier. The chiral algebra $A^G$ is called the \define{chiral orbifold} of~$A$~by~$G$.

What about the boundary condition ``$W$''? Pure gauge theory always admits a ``Dirichlet'' boundary condition. Indeed, Dijkgraaf--Witten theory for $\omega \in \H^3(G,\rU(1))$ has as its category of (bulk) line defects the Drinfel'd center $\cZ(\Vect^\omega[G])$, where $\Vect^\omega[G]$ denotes the grouplike fusion category whose simple objects are indexed by $G$ and whose associator is $\omega$ (see \cite[Example 2.3.8]{EGNO}). The pure-Dirichlet boundary condition is the topological boundary condition whose boundary line defects are in $\Vect^\omega[G]$ itself.  Suppose that $\omega = 0$ is trivialized. Such a choice of trivialization determines a second topological boundary condition via the equivalence $\cZ(\Vect^{\omega=0}[G]) \cong \cZ(\cat{Rep}(G))$; this second  boundary condition is ``pure Neumann.'' 
If one imposes both Dirichlet and Neumann boundary conditions on the same field, that field simply evaporates from the quantum field theory, and so the sandwich of $A^G$ with $\Vect^\omega[G]$ is simply $A$. On the other hand, imposing Neumann boundary conditions twice is the same as doing the path integral. By definition, the \define{full orbifold} $A\sslash G$ is the sandwich of $A^G$ and $\cat{Rep}(G)$; note that it can be constructed only when the anomaly $\omega$ has been trivialized.

There is a Morita equivalence of fusion categories $\Vect^{\omega=0}[G] \simeq \cat{Rep}(G)$ \cite[Example 3.15]{MR3077244}, which provides an invertible topological defect between Neumann and Dirichlet boundary conditions and hence between $A$ and $A\sslash G$.

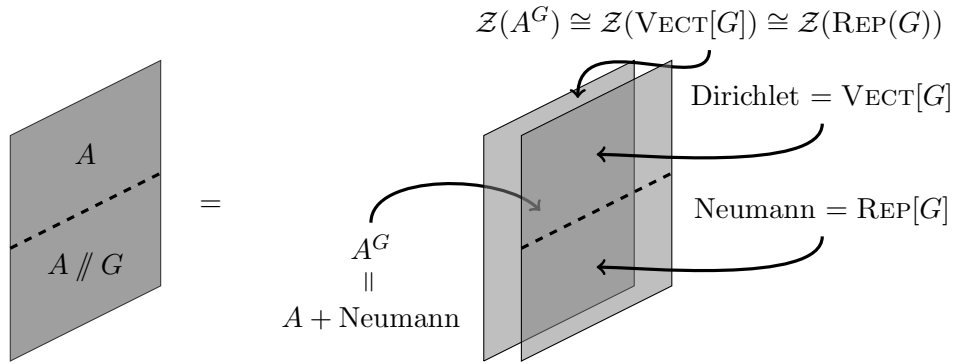
$$
\begin{tikzpicture}[baseline=(A)]
  \path[draw=black,fill=gray,opacity=.75] (0,0) -- (2,1) -- (2,4) -- (0,3) -- cycle;
  \path[draw,very thick,dashed] (0,1.5) -- (2,2.5);
  \path (1,2.75) node {$A$};
  \path (1,1.25) node {$A\sslash G$};
  \path (0,2) coordinate (A);
\end{tikzpicture}
\quad = \quad
\begin{tikzpicture}[baseline=(A)]
  \draw[->,very thick] (-1.5,1) node (V) {$\begin{array}{c} A^G \\ \rotatebox{90}{$=$} \\ A + \text{Neumann}\end{array}$} (V) .. controls +(0,1.5) and +(-.5,.5) .. (.75,2);
  \path[fill=gray,opacity=.5] (0,0) --  (2,1) -- (2,4) -- (0,3) -- (0,0);
  \path[draw=black] (0,0) --  (2,1) -- (2,4) -- (0,3) --  (0,0);
  \draw[->,very thick] (3,4.5) node (Z) {$\cZ(A^G) \cong \cZ(\Vect[G]) \cong \cZ(\cat{Rep}(G))$} (Z) .. controls +(0,-1) and +(0,1) .. (1.25,3.5);
  \path[fill=gray,opacity=.5] (.5,0) -- (2.5,1) -- (2.5,4) -- (.5,3) -- cycle;
  \path[draw=black] (.5,0) -- (2.5,1) --  (2.5,4) -- (.5,3) -- cycle;
  \path[draw,very thick,dashed] (.5,1.5) -- (2.5,2.5);
  \path (0,2) coordinate (A);
  \draw[->,very thick] (4.5,3.5) node (W) {Dirichlet $= \Vect[G]$} (W) .. controls +(0,-1) and +(.5,0) .. (1.5,2.75);
  \draw[->,very thick] (4.5,2) node (W) {Neumann $= \cat{Rep}[G]$} (W) .. controls +(0,-1) and +(.5,0) .. (1.5,1.25);
\end{tikzpicture}
$$

The description of $A^G$ as a conformal boundary condition for $\cZ(\Vect^\omega[G])$ explains the relationship between anomalies and the multipliers for twisted-twining genera \cite[Section 3]{MR3108775}. Indeed, following \cite{MR2500561}, let $\cM = \cM_{1}$ denote the moduli stack of elliptic curves and let $\cM_G$ denote the moduli stack of elliptic curves equipped with a principal $G$-bundle, and let $L_\omega$ denote the line bundle thereon constructed from $\omega$. The Hilbert space that the Dijkgraaf--Witten theory $\cZ(\Vect^\omega[G])$ assigns to an elliptic curve $E$ is the fiber over $E \in \cM$ of the pushforward of $L_\omega$ along $\cM_G \to \cM$. The ``twisted-twining genera'' are the (genus-one) conformal blocks of $A^G$; abstract nonsense of boundary field theories says that they are  sections of $L_\omega$. 
Consider the point $x  \in \cM_G$ corresponding to the elliptic curve $E_\tau = \bC / (\bZ \oplus \tau \bZ)$ equipped with the $G$-bundle that is trivial in the $\bZ$ direction and has monodromy $g\in G$ in the $\tau\bZ$-direction. The corresponding twisted-twining genus is $g$-twined but untwisted; as a function of $q = \exp(2\pi i \tau)$, it is the graded character of the action of $g$ on $A$. Suppose $g$ has order $n$. Then  the based fundamental group $\pi_1(\cM_G,x)$ for this point $x \in \cM_G$ contains a subgroup generated by $ST^nS^{-1}$, where $S,T \in \mathrm{SL}(2,\bZ)$ are the standard generators. As explained in \cite[\S 3.3]{MR3108775}, the monodromy along this loop in the line bundle $L_\omega$ --- the ``multiplier'' for the action of $ST^nS^{-1}$ on the character of $g$ --- equals the value of $\omega$ when evaluated on the $3$-cycle represented by $\sum_{k=0}^{n-1} g \otimes g^k \otimes g$.

\begin{example}\label{example permutation}
  Suppose that $V$ is a holomorphic conformal field theory and consider the permutation action of the symmetric group $G = S_n$ on $A = V^{\otimes n}$. It is commonly believed (see e.g.\ the conjecture of M\"uger \cite[Appendix 5, Conjecture 6.3]{MR2674592} or the remark on p.\ 2 of \cite{Davydov2013}) that such permutation actions are nonanomalous. This belief is false. The correct statement is that the anomaly is $c p_1$, where $c$ is the central charge of $V$ and $p_1 \in \H^3(S_m,\rU(1))$ is the first Pontryagin class of the permutation representation of $S_n$.  
  When $n \geq 6$, the class $p_1$ has order $12$ in $\H^3(S_n,\rU(1)) \cong \bZ_{12} \oplus \bZ_2^2$ \cite[Theorem~7.1]{MR878978} (when $n\leq 6$, the order of $p_1$ divides but can be less than $12$), and, for bosonic conformal field theories, $c$ is divisible by $8$, so the anomaly has order dividing $3$. It is nontrivial, for example, for $V = E_{8,1} = V_{E_8}$, the lattice conformal field theory corresponding to the $E_8$ lattice.
  
  To see that the anomaly can be nonzero,
  consider the action of an $n$-cycle $(1\dots n) \in S_n$. Let $\chi_V(q) = \tr( q^{L_0 - c/24}; V)$ denote the graded dimension of $V$. A straightforward combinatorics exercise shows:
  $$ \tr( q^{L_0 - nc/24} g; V^{\otimes n}) = \chi_V(q^n).$$
  Note that the central charge of $V^{\otimes n}$ is $nc$ if $V$ has central charge $c$. If $V$ is bosonic and holomorphic, then $S \in \mathrm{SL}(2,\bZ)$ fixes $\chi_V$ but $T$ acts on $\chi_V$ with eigenvalue $\exp(2\pi i c/24)$. Recalling that $q = \exp(2\pi i \tau)$ and that $S : \tau \mapsto -1/\tau$ and $T : \tau \mapsto \tau+1$, it is not hard to compute that $ST^nS^{-1}$ acts on $\tr( q^{L_0 - nc/24} g; V^{\otimes n})$ with eigenvalue $\exp(2\pi i c/24)$.
  
  One can summarize with the following slogan: for permutation orbifolds, the ``gauge anomaly''~$\omega$ is precisely the ``gravitational anomaly'' $\exp(2\pi i c/24)$.
\end{example}

\begin{example}
  The story above is not special to dimensions two and three: what is special in those dimensions is the description of Neumann and Dirichlet boundary conditions as different Turaev--Viro presentations of the same Reshetikhin--Turaev theory. One dimensional lower, the analog of a holomorphic conformal field theory is a separable symmetric Frobenius algebra $A$ with trivial center $\cZ(A) = \bC$.
  Such $A$ provides the boundary observables for a boundary condition for the trivial two-dimensional theory; i.e.\ a one-dimensional field theory.
   Indeed, such $A$ is necessarily a matrix algebra $A = \operatorname{End}(\bC^n)$, and the Hilbert space for the  one-dimensional  field theory is $\bC^n$. 
   
   Suppose that a finite group $G$ acts on $A = \operatorname{End}(\bC^n)$ by algebra automorphisms. 
  The anomaly $\omega \in \H^2(G,\rU(1))$ is
   the obstruction to making this action inner --- the failure for this action to descend to an action of $G$ on $\bC^n$
   --- and the orbifold theory, if $\omega = 0$, should have as its Hilbert space the $G$-fixed subspace $(\bC^n)^G$.
  Assuming the action of $G$ on $A$ is faithful, for arbitrary $\omega$, $A^G$ is Morita equivalent to the $\omega$-twisted group algebra $\bC^\omega[G]$. In particular, $\cZ(A^G) = \cZ(\bC^\omega[G])$, and $\bC^\omega[G]$ defines defines the ``Dirichlet'' boundary condition in bulk-boundary story above. When $\omega$ has been trivialized, the ``pure Neumann'' boundary condition arises from the well-known isomorphism between the algebra $\cZ(\bC^{\omega=0}[G])$ of class functions  and the algebra $R(G)$ of characters.
\end{example}

   In high dimensions, not all finite group gauge theories are of Dijkgraaf--Witten type. Rather, they arise from integrating the gauge degrees of freedom in a symmetry protected topological phase of matter. SPT phases are classified not by ordinary cohomology but rather by some generalized cohomology theory, which (if one works only with bosonic theories) agrees with ordinary cohomology in low dimensions. Some discussion of the relationship between anomalies and SPT phases can be found for example in \cite{Wen2013,GJFa}.

\subsection{Mathematical foundations} \label{why conformal nets}

There are two popular mathematical formalisms for the algebras of observables in chiral conformal field theories. A \define{vertex operator algebra} is the ``Taylor expansion'' at a point of the chiral algebra of all observables. Living in the purely algebraic world of power series, vertex operator algebras are very good for explicit computation. A \define{chiral conformal net} encodes instead those observables that are ``smeared'' along intervals in the unit circle surrounding the expansion point. 
Chiral conformal nets come from the world of von Neumann algebras and algebraic quantum field theory.
The axioms of both formalisms are widely available and will not be reviewed here.

Both formalisms provide precise definitions for adjectives like ``rational'' and ``holomorphic,'' and it is widely believed but not yet known that the rational conformal field theories in the two formalisms agree. Many special cases of their agreement are known. First, as proved in \cite[Chapter~7]{BischoffThesis}, the two formalisms are equivalent in the case when one has a chiral conformal field theory $V$ of central charge $c \in \frac12\bZ$ together with an injection $\mathrm{Vir}_{1/2}^{\otimes 2c} \hookrightarrow V$, where $\mathrm{Vir}_{1/2}$ denotes the $c=\frac12$ Virasoro algebra. Such algebras are called \define{framed} following \cite{MR1618135}, which is both well-deserved and also dangerous, since they are defined on oriented, not framed, complex curves.
The conformal field theories we will use in this paper --- the Leech lattice CFT $V_\Lambda$ and the moonshine CFT $V^\natural$ --- are framed~\cite[Examples 2.9 and 3.8]{MR2263720}, and so results proved about them in one formalism can be moved to the other. Second, 
the main purpose of \cite{CKLW} is to formulate an analytic condition on a CFT called ``strong locality'' and to show that the two formalisms agree when the CFT is strongly local.   The moonshine CFT is strongly local~\cite[Theorem 8.15]{CKLW}, and there are no rational CFTs which are known  to fail to be strongly local.

So far, the formalism of conformal nets has provided more powerful structural and category-theoretic results about conformal field theories. In particular, the orbifold/anomaly story of the previous section can be made fully rigorous in that formalism. On the vertex algebra side, the final missing piece is to show that if $V$ is rational, so is $V^G$. The full result has been announced in various conferences, and the strongest version to appear in print is the main result of \cite{CarnahanMiyamoto} covering the case when $G$ is solvable. On the conformal net side, the existence of the anomaly in $\H^3(G,\rU(1))$, together with how it controls the orbifold problem for holomorphic conformal field theories, was announced but not proven in \cite[Corollary~3.6]{MR2730815}, and is physically justified in \cite{MR1003430,MR1128130}.

\begin{proposition}[{\cite[Corollary 3.6]{MR2730815}}] \label{prop.orbifolding works}
  Let $A$ be a holomorphic conformal net and $G$ a finite group acting faithfully on $A$, and let $A^G$ denote the $G$-fixed sub conformal net. Then there is an \define{anomaly} $\omega \in \H^3(G,\rU(1))$ with the properties that the modular category $\cZ(A^G)$ of vertex modules for $A^G$ is naturally equivalent to the Drinfel'd center $\cZ(\Vect^\omega[G])$ of the pointed fusion category $\Vect^\omega[G]$ with associator determined by $\omega$. Trivializations of $\omega$ correspond to choices of full orbifold $A\sslash G$.
\end{proposition} 

There does not seem to be a complete statement and proof of this result in the literature, but it follows immediately from the work \cite{MR1806798,MR1923177,MR3039775,MR3424476}:

\begin{proof}
  The main result of \cite{MR1923177} is a proof of this claim under the additional hypothesis that~$A^G$ is rational. Although that paper is phrased in terms of vertex algebras, the proof is essentially model-independent, and applies equally well in the world of chiral conformal nets. Rationality of~$A^G$, provided $A$ is rational and $G$ is finite, is the main result of \cite{MR1806798}. (In the conformal net literature, what we call ``rationality'' is usually referred to as ``complete rationality.'')
 
  The only model-dependent result needed in \cite{MR1923177} is a characterization of extensions $V \subset W$ of chiral algebras in terms of certain algebra objects in $\cZ(V)$ --- these are called ``\'etale algebras'' in \cite{MR3039775} and ``Q-systems'' in \cite{MR3424476}. On the vertex algebra side, this characterization is due to \cite{MR1936496}. On the conformal net side, it is the main result of \cite{MR3424476}; see also \cite{MR1332979}. In particular, \cite[Proposition~6.6]{MR3424476} applied to the extension $A^G \subset A$ implies that $\cZ(A^G)$ contains what in the language of \cite{MR3039775} is called a ``{Lagrangian algebra}.'' By \cite[Corollary~4.1]{MR3039775}, this Lagrangian algebra is equivalent to a choice of fusion category $\cC$ with $\cZ(A^G) = \cZ(\cC)$, and the arguments of \cite{MR1923177} verify that $\cC = \Vect^\omega[G]$. Finally, as remarked already in Section~\ref{physical picture}, each trivialization of $\omega$ produces an equivalence $\cZ(\Vect^{\omega=0}[G]) = \cZ(\cat{Rep}(G))$, and so a new Lagrangian algebra in $\cZ(A^G) = \cZ(\Vect^{\omega=0}[G])$, and so a new extension $A^G \subset A\sslash G$ with $A\sslash G$ holomorphic.
\end{proof}

The series \cite{BDH0,BDH3,MR3439097,BDH4,MR3656522} provides a context in which the entire story of Section~\ref{physical picture} can be mathematically realized, and in particular provides a clean interpretation of the anomaly $\omega$. Specifically, that series builds a symmetric monoidal 3-category $\cat{ConfNet}$ whose objects are rational conformal nets.  $\cat{ConfNet}$ is a higher-categorical analog of the symmetric monoidal bicategory $\cat{Alg}$, first constructed in \cite{MR0220789}, whose objects are associative algebras and morphisms are bimodules.  Given a rational conformal net $V$, its modular tensor category $\cZ(V)$ of vertex modules arises in $\cat{ConfNet}$ as the 1-category of endo-2-morphisms of the identity 1-morphism from $V$ to itself. (As in Section~\ref{physical picture}, the notation $\cZ(V)$ is chosen in analogy with the cases of associative algebras and fusion categories. For example, if $V \in \cat{Alg}$ were an associative algebra, then its center $\cZ(V)$ would arise as the algebra of endo-2-morphisms of the identity 1-morphism on~$V$.)
 According to \cite[Theorem 2]{BDH0}, every rational conformal net is 3-dualizable, and so defines via \cite{Lur09} a fully-extended three-dimensional TFT valued in $\cat{ConfNet}$. By design, this TFT is the Reshetikhin--Turaev theory determined by $\cZ(V)$.
 
 Holomorphic conformal nets are invertible in $\cat{ConfNet}$ \cite[Theorem 3.23]{Tring}. In a symmetric monoidal category, all invertible objects have canonically-identified spaces of endomorphisms. Thus an action of $G$ on a holomorphic conformal net determines a (homotopy class of) map(s) $\omega : \mathrm{B} G \to \mathrm{B} \Aut_{\cat{ConfNet}}(\mathrm{TrivialObject})$. The latter is a $K(\rU(1),3)$, via a bosonic version of \cite[Theorem 3.19]{Tring}. (Specifically: the paper \cite{Tring} works with $\bZ/2$-graded conformal nets; if, following \cite{BDH0,BDH3,MR3439097,BDH4,MR3656522}, one works instead with bosonic conformal nets, \cite[Theorem 3.19]{Tring} identifies the endomorphism 2-category of any invertible conformal net with the Morita 2-category of von Neumann algebras, which has just one isomorphism class of invertible object and just one isomorphism class of invertible 2-morphism. $\cat{ConfNet}$ is a 3-category, and so $\mathrm{B} \Aut_{\cat{ConfNet}}(\mathrm{TrivialObject})$ should be read as a homotopy 3-type rather than a specific space.) In particular, maps into $\Aut_{\cat{ConfNet}}(\mathrm{TrivialObject})$ are naturally identified with classes in $\H^3(-,\rU(1))$, and the map $\omega$ is the anomaly.

\subsection{T-duality for finite groups} \label{t-duality}
Suppose now that the group $G$ acting on a holomorphic conformal field theory $V$ has a cyclic normal subgroup of order $n$, so that $G$ is an extension of shape $G = n.J$ for some finite group $J = G/n$. (We continue to follow the ATLAS \cite{ATLAS} and let ``$n$'' denote a cyclic group of order~$n$.) In practice, such groups arise as centralizers and normalizers of cyclic subgroups $n \mono \Aut(V)$.
Suppose further that the anomaly $\omega \in \H^3(G,\rU(1))$ of the action of $G$ on $V$ trivializes when restricted to the subgroup $n$. (Since $n$ is cyclic, $\H^2(n,\rU(1)) = 0$, and so there is no data in the choice of trivialization.) By the above discussion, one can produce the orbifold conformal field theory $V \sslash  n$. An interesting question to ask is: what group acts on $V\sslash n$, and with what anomaly?

The special case $\omega = 0$ of this question is answered in \cite{BT2017}.  We will give the answer under the assumption $\H^1(J,n) = \H^1(J,\widehat{n}) = 0$, where $\widehat{n} = \hom(n,\rU(1))$ denotes the Pontryagin dual cyclic group to $n$. (We will distinguish~$n$ from~$\widehat{n}$ in what follows, not least to accommodate the case that $n \subset G$ is normal but not central, in which case the $J$-actions on~$n$ and~$\widehat{n}$ might not be isomorphic.)  This assumption holds in all the cases we will care about in Section~\ref{small primes}. It is convenient because it means that there is no essential ambiguity in choosing a cocycle for a class in $\H^1(J,n)$ or $\H^1(J,\widehat{n})$. Indeed, any two cocycles will then differ by the derivative of a \emph{unique} cohomology class of $1$-cochains. (Of course, two not-necessarily-closed cochains are \define{cohomologous} if they differ by an exact cochain.)

Any group extension provides a Lyndon--Hochschild--Serre (LHS) spectral sequence. The $E_2$-page of the spectral sequence, in the case of the extension $G = n.J$, is $\H^\bullet(J, \H^\bullet(n,\rU(1)))$, and the sequence converges to $\H^\bullet(G,\rU(1))$. 
In total degree $3$ the LHS spectral sequence writes $\H^3(G,\rU(1))$ as an extension 
$$\H^3(G,\rU(1)) = (\text{quotient of $\H^3(J,\rU(1))$}).(\text{sub of $\H^2(J,\widehat{n})$}).\H^3(n,\rU(1)),$$
where we have used the isomorphisms $\H^1(n,\rU(1)) = \widehat{n}$, $\H^2(n,\rU(1)) = 0$, and $\H^1(J,\widehat{n}) = 0$ to simplify the formulas.
The quotients and subs are, of course, cut out by the higher-degree differentials in the spectral sequence. The map $\H^3(G,\rU(1)) \to \H^3(n,\rU(1))$ takes $\omega \mapsto \omega|_n$. By assumption, $\omega|_n = 0$. The class $\omega$ therefore determines an element $\alpha \in (\text{sub of $\H^2(J,\widehat{n})$})$.

$$\begin{tikzpicture}
  \matrix{
  \node{$\H^0(J,\Sym^2(\widehat{n}))$}; \\
  \node{$0$}; & \node{$0$}; & \node{$0$}; \\
  \node{$*$}; & \node{$0$}; & \node (a) {$\H^2(J,\widehat{n})$}; \\
  \node{$*$}; & \node{$*$}; & \node{$*$}; & \node {$\H^3(J,\rU(1))$}; & \node (b) {$\H^4(J,\rU(1))$}; \\
  };
  \draw[->] (a) -- node[auto] {$\scriptstyle \d_2$} (b);
\end{tikzpicture}$$

To give the remaining datum of $\omega$ requires understanding the $\d_2$ differential in the spectral sequence.
Let $\kappa \in \H^2(J,n)$ classify the extension $G = n.J$. 
Write $\langle,\rangle$ for the pairing $\widehat{n} \otimes n \to \rU(1)$
 and  $\langle \cup \rangle$ for the induced pairing $\H^i(J,n) \otimes \H^j(J,\widehat{n}) \overset\cup\to \H^{i+j}(J,n \otimes \widehat{n}) \overset{\langle,\rangle}\to H^{i+j}(J,\rU(1))$.
Then
$$ \d_2(\alpha) = \langle \alpha \cup \kappa\rangle \in \H^4(J,\rU(1)).$$
Choose cocycles for $\alpha$ and $\kappa$. Then the final datum of $\omega$ is a cohomology class of solutions $\beta \in \C^3(J,\rU(1))$ to 
$$\d \beta = \langle \alpha \cup \kappa\rangle.$$

The above equation for $\beta$ is manifestly symmetric under switching the roles of $\alpha$ and $\kappa$. We can therefore define the \define{T-dual} of the pair $\bigl(G = n.J, \omega \in \H^3(G,\rU(1))\bigr)$ to be the pair $(\widehat{G},\widehat{\omega})$ where $\widehat{G} = \widehat{n}.J$ is classified by $\alpha$ and where $\widehat{\omega} \in \H^3(\widehat{G},\rU(1))$ is the cohomology class assembled from the data $(\kappa,\beta)$. This ``finite group T-duality'' is obviously an involution. The name comes from thinking of the cyclic group $n$ as a ``finite circle,'' the extension $G = n.J$ as a ``finite circle bundle,'' the Pontryagin dual group $\widehat{n}$ as the ``dual circle,'' the extension $\widehat{G}$ as the ``dual circle bundle,'' and the anomaly $\omega\in \H^3(G,\rU(1))$ as the ``Kalb--Ramond field.'' 

\begin{proposition}\label{t duality prop}
  With notation as above, suppose that $G = n.J$ acts on a holomorphic conformal field theory $V$ with anomaly $\omega$. Then the T-dual group $\widehat{G}$ acts on the orbifold $V\sslash n$ with T-dual anomaly $\widehat{\omega}$.
\end{proposition}

\begin{proof}
  We first reformulate T-duality in terms of fusion categories. The requirements $G = n.J$ and $\omega|_n = 0$ are equivalent to saying that $\Vect^\omega[G]$ is a \define{$J$-graded extension of $\Vect[n]$}.
The main theorem of~\cite{MR2677836} (see also \cite{MR2511638}) classifies $J$-graded extensions of $\Vect[n]$.  The first datum in that classification is a map $J \to \rO(n \oplus \widehat{n})$, where the latter denotes the group of transformations that are orthogonal for the canonical  pairing. The actions corresponding to group extensions $G = n.J$ are precisely those where the map $J \to \rO(n \oplus \widehat{n})$ factors through the map $J \to \mathrm{GL}(n) \cong \mathrm{GL}(\widehat{n})$ giving the outer action of $J$ on $n$.   The next datum is the class $\kappa \oplus \alpha \in \H^2(J,n\oplus \widehat{n})$.
  This datum cannot be arbitrary: the \define{obstruction} $\langle \alpha \cup \kappa \rangle \in \H^4(J,\rU(1))$ must vanish. The third and final datum needed is a cohomology class of solutions $\beta$ to $\d \beta = \langle \alpha \cup \kappa\rangle$ (for some choices of cocycles of $\kappa$ and $\alpha$). 
  The appendix to \cite{MR2677836} identifies these data with the components of the expansion of $\omega$ via the LHS spectral sequence for $\H^\bullet(n.J,\rU(1))$. 
  The classification is manifestly invariant under Pontryagin duality $n \leftrightarrow \widehat{n}$. Indeed, as observed in \cite{MR2677836}, the classification of $J$-graded extensions of $\Vect[n]$ depends only on the Morita equivalence class of the fusion category $\Vect[n]$, and Pontryagin duality corresponds to the Morita equivalence $\Vect[n] \simeq \cat{Rep}(n) = \Vect[\widehat{n}]$ of \cite[Example 3.15]{MR3077244}.
  
  The fusion-categorical reformulation in hand, it is straightforward to confirm that if $G = n.J$ acts on $V$ with anomaly $\omega$, then $\widehat{G}$ acts on $V\sslash n$ with anomaly $\widehat{\omega}$. Indeed, consider $V$ and $V\sslash n$ each as extensions of $V^G$. As in the proof of Proposition~\ref{prop.orbifolding works}, the two extensions correspond to two different Lagrangian algebras in $\cZ(V^G) = \cZ(\Vect^\omega[G])$, or equivalently two different fusion categories with center $\cZ(V^G)$. One is simply $\Vect^\omega[G]$. The other is easily seen to be $\Vect^{\widehat{\omega}}[\widehat{G}]$, by inspecting the construction from \cite{MR2677836}. This forces $V^G$ to be precisely the fixed points of an action of $\widehat{G}$, with anomaly $\widehat{\omega}$, on $V\sslash n$.
\end{proof}

Since T-duality is involutive, one immediately finds the following examples. Suppose $g \in \Aut(V)$ is an element of order $n$, generating the subgroup $\langle g \rangle \cong n$, and suppose that this subgroup acts nonanomalously on $V$.
The \define{T-dual element} $\widehat{g}$ of $g$ is the generator of the dual $\widehat{n}$ acting on $V\sslash \langle g\rangle$ (normalized so that $\langle \widehat{g},g\rangle = \exp(2\pi i/n) \in \rU(1)$).
 Then the centralizer $C(g) \subset \Aut(V)$ of $g$ is T-dual to the centralizer $C(\widehat{g}) \subset \Aut(V\sslash \langle g \rangle)$, and the normalizer $N(\langle g\rangle) \subset \Aut(V)$ of the subgroup $\langle g\rangle$ is T-dual to $N(\langle \widehat{g} \rangle) \subset \Aut(V\sslash \langle g \rangle)$.

\subsection{Super case} \label{super case}

Although it will not play a role in this paper, it's worth mentioning for completeness the description of orbifolds of possibly-fermionic conformal field theories. If $V$ is allowed to have fermions, then the anomaly of an action of $G$ on $V$ is not a class in $\H^3(G,\rU(1))$. It is instead a class in the ``extended supercohomology'' of~\cite{WangGu2017}, which, as explained in \cite{GJFa}, is the generalized cohomology theory corresponding to the spectrum $S = \cat{SuperAlg}_\bC^\times$ with homotopy groups $\pi_{-2} S = \pi_{-1} S = \bZ_2$ and $\pi_0 S = \rU(1)$ and k-invariants $\Sq^2 : \bZ_2 \to \bZ_2$ and $\exp(\pi i \Sq^2) : \bZ_2 \to \rU(1)$. (The two possible j-invariants in this case give equivalent spectra.)
A cocycle $\omega \in \Z^3(G,S)$ consists of cochains $\omega_1 \in \C^1(G,\bZ_2)$, $\omega_2 \in \C^2(G,\bZ_2)$, and $\omega_3 \in \C^3(G,\rU(1))$ solving $\d \omega_1 = 0$, $\d \omega_2 = \Sq^2\omega_1$, and $\d \omega_3 = \exp\bigl(\pi i \Sq^2\omega_2 + (\text{term involving }\omega_1)\bigr)$.
Let $S'$ denote the spectrum with homotopy groups $\pi_{-1} S' = \pi_0 S' = \bZ_2$ and k-invariant $\Sq^2$. Then $\H^3(G,S)$ fits into a long exact sequence like
$$ \dots \to \H^3(G,\rU(1)) \to \H^3(G,S) \to \H^2(G,S') \to \dots. $$
The map $ \H^3(G,\rU(1)) \to \H^3(G,S)$ corresponds to considering a bosonic conformal field theory as a possibly-fermionic one.

One can understand the anomaly in various ways. In terms of the induced action of $G$ on the trivial 3D theory $\cZ(V)$, the spectrum $S$ appears as a connective cover of the Pontryagin dual to Spin Cobordism.
The anomaly $\omega \in \H^3(G,S)$ also has meaning in terms of the \emph{super} fusion category of solitons on $V$ that restrict to honest modules on $V^G$. As in the even case, there is a ``unique'' irreducible $g$-graded soliton (aka twisted sector) $V_g$ for each $g\in G$, where ``unique'' means up to possibly-odd isomorphism. Then $\omega_1(g) \in \bZ_2$ measures whether whether the endomorphism algebra of $V_g$ is $\bC$ or $\mathrm{Cliff}_\bC(1)$ (objects of the latter type are called \define{Majorana}), and $\omega_2(g_1,g_2) \in \bZ_2$ measures whether the isomorphism $V_{g_1}\otimes V_{g_2}\cong V_{g_1g_2}$ is even or odd. Finally, $\omega_3$ is the associator familiar from the non-super case.

As in the bosonic case, if $G$ acts on $V$ with anomaly $\omega$, then orbifolds $V\sslash G$ depend on a choice of trivialization of $\omega$. Unlike in the bosonic case, if $G$ is cyclic of even order, there is a choice in the trivialization. To trivialize $\omega$, one first must trivialize $\omega_1 \in \H^1(G,\bZ_2)$; this either trivializes or it doesn't. Such a choice makes $\omega_2$, which originally lives in a torsor for $\H^2(G,\bZ_2)$, into an honest cohomology class; one must then trivialize $\omega_2$, and there are $\H^1(G,\bZ_2)$-many ways to do so. This group is a copy of $\bZ_2$ when $G = \bZ_{2n}$, whereas in the bosonic case there was no ``$\omega_2$'' needing trivialization. Finally, any such choice makes $\omega_3$ into an honest cohomology class, and one must trivialize it --- there are $\H^2(G,U(1))$ many choices for the trivialization, and this group vanishes for $G$ cyclic. 

Super T-duality is clearest when $\omega_1 = 0$. Then $\omega = (\omega_2,\omega_3)$ is a class in the restricted supercohomology of \cite{GuWen}. More pedantically, a cocycle representative of $\omega$ consists of a cocycle $\omega_2 \in \Z^2(G,\bZ_2)$ and a cochain $\omega_3 \in \C^3(G,\rU(1))$ solving $\d\omega_3 = (-1)^{\Sq^2 \omega_2} = (-1)^{\omega_2 \cup \omega_2}$, which in  the notation of Section~\ref{t-duality} is $\d \omega_3  = \langle \omega_2\cup\omega_2\rangle$.

The super fusion category $\cat{SuperVect}^\omega[G]$ has an underlying non-super fusion category. It is again group-like of shape $\Vect^{\tilde \omega}[\tilde G]$, where $\tilde G = 2.G$ is the extension classified by $\omega_2 \in \H^2(G,\bZ_2)$. The associator is $\tilde \omega = \omega_2 + \omega_3 \in \H^3(\tilde G,\rU(1))$; it exists because of the equation $\d\omega_3 = \langle \omega_2\cup\omega_2\rangle$, just as in Section~\ref{t-duality}. The pair $(\tilde G,\tilde \omega)$ is called the \define{bosonic shadow} of $(G,\omega)$ in \cite{BGK2016}.

Suppose now that $G = n.J$ and that $\omega|_n = 0$ (meaning of course that a trivialization has been chosen). Construct its bosonic shadow $\tilde G = 2.n.J$. Since $\omega|_n = 0$, the subgroup $2.n = \tilde n \subset \tilde G$ is a direct product: $\tilde G = (2\times n).J = n.2.J = n.\tilde J$. Applying the bosonic T-duality from Section~\ref{t-duality} produces a new group $\widehat{\tilde{G}} = \widehat{n}.\tilde J = \widehat{n}.2.J$. One  has $\widehat{n}.2 = \widehat{n} \times 2$, and so we can write $\widehat{\tilde{G}} = 2.\widehat{n}.J$. Define $\widehat{G} = \widehat{\tilde{G}}/2 = \widehat{n}.J$. To complete the construction, one checks that the T-dual bosonic anomaly $\widehat{\tilde{\omega}} \in \H^3(2.\widehat{G},\rU(1))$ splits as $\widehat{\omega_2} + \widehat{\omega_3}$ with $\d\widehat{\omega_3} = \langle \widehat{\omega_2}\cup\widehat{\omega_2}\rangle$; then $\widehat{\omega} = (\widehat{\omega_2},\widehat{\omega_3})$ is the dual super anomaly for $\widehat{G}$.

\begin{example}\label{example fer 24}
An important class of fermionic conformal field theories are the free Majorana fermion models $\Fer(n)$ of central charge $c = \frac n 2$. The automorphism group of $\Fer(n)$ is $\rO(n)$. According to \cite{MR2742433}, the anomaly of an action $G \to \rO(n)$ is precisely the \define{string} obstruction --- the obstruction to lifting the action to a map $G \to \mathrm{String}(n)$, which by definition is the $3$-connected cover of $\rO(n)$. $\mathrm{String}(n)$ fits into a tower $\dots \to \mathrm{String}(n) \to \mathrm{Spin}(n) \to \mathrm{SO}(n) \to \rO(n)$. The obstruction $\omega$ can be understood in pieces by trying to lift into the intermediate groups, and these pieces match the ones above: $\omega_1$ is the first Stiefel--Whitney class, obstructing the lift from $\rO(n)$ to $\mathrm{SO}(n)$; $\omega_2$ is the second Stiefel--Whitney class, obstructing the lift from $\mathrm{SO}(n)$ to $\mathrm{Spin}(n)$; and $\omega_3 = \frac{p_1}2$ is the first fractional Pontryagin class, obstructing the lift from $\mathrm{Spin}(n)$ to $\mathrm{String}(n)$.
In general, for an action of $G$ on a fermionic conformal field theory $V$, trivializations of $\omega_1$ and $\omega_2$ can be thought of as choices of ``orientation'' and ``spin structure.''

Consider the central subgroup $\bZ_2 \subset \rO(n) = \Aut(\Fer(n))$ switching the signs of all fermions simultaneously. Its string obstruction --- its anomaly --- vanishes exactly when $n$ is divisible by $8$. Set $n=8k$ and $G = 2.J \subset \mathrm{SO}(n)$, so that we can ignore $\omega_1$. What is the T-dual group $\widehat{G}$ acting on $\Fer(n)\sslash \bZ_2$? To make the question well-posed we need to choose a trivialization of $\omega_2 |_{\bZ_2}$. The choices are a torsor for $\H^2(\bZ_2,\bZ_2) = \bZ_2$. The two choices correspond to the two nontrivial double covers other than $\mathrm{SO}(n)$ of $\mathrm{PSO}(n)$. These double covers,  called $\mathrm{SO}^+(n)$ and $\mathrm{SO}^-(n)$, are the images of $\mathrm{Spin}(n)$ under the two spin-module actions. Choose either $\mathrm{SO}^\pm(n)$. Then $\widehat{G}$ is nothing but the preimage in $\mathrm{SO}^\pm(n) = 2.\mathrm{PSO}(n)$ of $J \subset \mathrm{PSO}(n)$.

Although the groups $\mathrm{SO}^\pm(n)$ are isomorphic --- they are related by a reflection in $\rO(n)$ --- the preimages $\widehat{G} = 2.J$ may not be. A main example is $G = \Co_0 = 2.\Co_1 \subset \mathrm{SO}(24)$. Then one of the preimages is again isomorphic to $\Co_0$ and the other is isomorphic to $2\times \Co_1$. 
These choices show up when studying the ``super moonshine'' modules of \cite{MR2352133,MR3376736}. Both papers study a fermionic CFT isomorphic to $\Fer(24) \sslash  \bZ_2$, with automorphism group $\mathrm{SO}^\pm(24)$. According to \cite{MR2352133}, this CFT admits $\cN = 1$ supersymmetry, and as a supersymmetric conformal field theory its automorphism group is $\Co_1$. According to \cite{MR3376736}, the genus-zero phenomena central to Monstrous Moonshine appear when studying instead a faithful $\Co_0$ action on $\Fer(24) \sslash  \bZ_2$. Fix a copy of $\Co_0 \subset \mathrm{SO}(24)$. One can think of the two ``super moonshine'' CFTs as two different choices of the orbifold $\Fer(24) \sslash  \bZ_2$. If, when defining $\Fer(24) \sslash  \bZ_2$, we use the spin structure for the $\bZ_2$-action that extends to all of $\Co_0$, the T-dual group is $\Co_0$, and so we get the version from \cite{MR3376736}; if we use the other spin structure, the T-dual group is $2 \times \Co_1$, of which only $\Co_1$ preserves the $\cN=1$ supersymmetry, and we get the version from \cite{MR2352133}.

If follows from Theorem~\ref{conway theorem} that the $\Co_0$-action on the CFT from \cite{MR3376736} has anomaly of order $24$ in $\H^3(\Co_0,\rU(1)) \cong \bZ_{24}$. The anomaly for the $\Co_1$-action on the CFT from \cite{MR2352133} is honestly super: it is a class in the  supercohomology of $\Co_1$. This is a copy of $\bZ_{24} = \bZ_{12}.\bZ_2 = \H^3(\Co_1,\rU(1)).\H^2(\Co_1,\bZ_2)$, and the anomaly again has order $24$.
\end{example}

\begin{example}
Another class of examples of fermionic orbifolds comes from lattice conformal field theories. If~$L$ is an even unimodular lattice and $L' \subset L$ has index $2$, then $(L')^*/L$ is isomorphic to either $\bZ_4$ or to $\bZ_2^2$. The former case corresponds to an anomalous action of $\bZ_2$ on the lattice conformal field theory $V_L$, and the latter occurs when the anomaly is trivial. In the latter case, there is a unique even lattice other than $L$ that contains $L'$; this uniqueness reflects the uniqueness in bosonically trivializing the anomaly.  But there is also an odd lattice containing $L'$ in index $2$, which comes from the other choice of fermionic trivialization.
\end{example}

\section{Proof of Theorem~\ref{moonshine theorem}} \label{M}

We now prove Theorem~\ref{moonshine theorem}, which asserts that the action of the Monster group $\bM$ on its natural representation $V^\natural$ has an anomaly $\omega^\natural$ of order exactly $24$.  The  conformal field theoretic inputs we use are:
\begin{enumerate}
  \item the existence of $V^\natural$ and $\bM = \Aut(V^\natural)$, and hence of the anomaly $\omega^\natural \in \H^3(\bM,\rU(1))$;
  \item isomorphisms $V^\natural \sslash  \bZ_p \cong V_\Lambda$ for $p = 2,3,5,7,13$, and hence T-dualities between certain subgroups of $\bM$ and $\Aut(V_\Lambda) = \widehat{\Lambda}.\Co_0$, where $\Lambda$ denotes the Leech lattice, $\widehat{\Lambda}$ is its Pontryagin dual torus, and $V_\Lambda$ is the Leech lattice conformal field theory.
\end{enumerate}

When calculating cohomology classes of finite groups, one may proceed prime-by-prime, because of the following standard result.

\begin{lemma}\label{transfer-restriction}
 Let $G$ be a finite group. Then $\H^k(G,\rU(1))$ is finite abelian for $k\geq 1$, and so splits as $\H^k(G,\rU(1)) = \bigoplus_p \H^k(G,\rU(1))_{(p)}$ where the sum ranges over primes $p$ and $\H^k(G,\rU(1))_{(p)}$ has order a power of $p$.  Fix a prime $p$ and suppose that $S \subset G$ is a subgroup such that $p$ does not divide the index $|G|/|S|$, i.e.\ such that $S$ contains the $p$-Sylow of $G$. Then the {restriction} map $\alpha \mapsto \alpha|_S : \H^k(G,\rU(1)) \to \H^k(S,\rU(1))$ is an injection onto a direct summand.\qed
\end{lemma}

The Sylow subgroups of $\bM$ are known (and listed for example in the ATLAS \cite{ATLAS}).  For reference sake, the order of the Monster is
$$ |\bM| = 2^{46} \cdot 3^{20} \cdot 5^9 \cdot 7^6 \cdot 11^2 \cdot 13^3 \cdot 17 \cdot 19 \cdot 23 \cdot 29 \cdot 31 \cdot 41 \cdot 47 \cdot 59 \cdot 71.$$
We consider the primes $17$--$71$, and also the prime $11$, as ``large,'' since they can be dispensed with quickly in Section~\ref{large primes}. The primes $3$, $5$, $7$, and $13$ are ``small,'' and we handle them using T-duality in Section~\ref{small primes}. We also handle the prime $2$ using T-duality, but it is sufficiently complicated that we give it its own Section~\ref{monster 2}.

\subsection{The ``large'' primes \emph{p} = 11, 17, 19, 23, 29, 31, 41, 47, 59, 71} \label{large primes}

\begin{lemma}
  $\H^3(\bM,\rU(1))_{(p)} = 0$ for $p\geq 17$ or $p=11$.
\end{lemma}

\begin{proof}
Let $p \geq 17$ be a prime dividing $|\bM|$. Then the $p$-Sylow in $\bM$ is cyclic and contained in a subgroup of shape $p:n$ for $n = \frac{p-1}2 > 2$. It is easy to see that $\H^3(p:n,\rU(1))_{(p)} = 0$. Indeed, it is a direct summand of $\H^3(p,\rU(1)) \cong \bZ_p$, and it is fixed by the $n$-action. The automorphism $g \mapsto g^a$ of $\bZ_p$ acts on $\H^3(\bZ_p,\rU(1))$ by multiplication by $a^2$. But $n \subset \Aut(p) = \bZ_p^\times$ contains elements other than $\pm 1$, and so acts nontrivially on $\H^3(p,\rU(1))$, and so leaves only the $0$ subgroup fixed.

For $p=11$, the $11$-Sylow in $\bM$ is contained in a subgroup of shape $(11:5)^{\times 2}$, and $\H^3\bigl((11:5)^{\times 2},\rU(1)\bigr)_{(11)}$ vanishes by K\"unneth's formula.
\end{proof}

\subsection{The ``small'' primes \emph{p} = 3, 5, 7, 13} \label{small primes}

Let $\Lambda$ denote the Leech lattice, $\Co_0 = \rO(\Lambda)$ its orthogonal group, $\widehat{\Lambda} = \hom(\Lambda,\rU(1)) \cong \rU(1)^{24}$ its dual torus, and $V_\Lambda$ its lattice conformal field theory. As is true for any lattice, the construction $\Lambda \leadsto V_\Lambda$ makes manifest a subgroup of $\Aut(V_\Lambda)$ of shape $\widehat{\Lambda}.\rO(\Lambda)$; this result was first announced in \cite{MR820716} and is surveyed nicely in \cite[Section 5.3]{MollerThesis}. In the Leech case, $\Aut(V_\Lambda)$ is in fact equal to $\widehat{\Lambda}.\rO(\Lambda)$ by \cite[Theorem 2.1]{MR1745258} and the fact the Leech lattice $\Lambda$ has no roots (so that the connected Lie group group called ``$N$'' in \cite{MR1745258} is just $\widehat{\Lambda}$).
In the sense of vertex algebras, the following result is definitional for $p=2$, conjectured in \cite{MR996026,MR1165184} for  $p\in \{3,5,7,13\}$, announced for $p=3$ in \cite{MR1302011}, and proved in full in \cite{AbeLamYamada}:

\begin{proposition}[\cite{AbeLamYamada}] \label{ALYtheorem}
Let $\bZ_p \subset \widehat{\Lambda}.\Co_0 = \Aut(V_\Lambda)$ be a subgroup of prime order $p$ such that its image in $\Co_0 = \rO(\Lambda)$ acts on $\Lambda$ without fixed points. Then $p$ is one of $2$, $3$, $5$, $7$, or $13$ and every fixed-point-free automorphism of $\Lambda$ of order $p$ lifts to an order-$p$ automorphism of $V_\Lambda$. The action of $\bZ_p$ on $V_\Lambda$ is nonanomalous and the orbifold $V_\Lambda \sslash  \bZ_p$ is isomorphic to $V^\natural$.
\end{proposition}

\begin{proof}
  Since the statement is verified in the sense of vertex algebras in \cite{AbeLamYamada}, it suffices to transfer the claim to the conformal net setting. The work on ``framed CFTs'' from \cite[Chapter~7]{BischoffThesis} verifies analytically the isomorphism $(V_\Lambda) \sslash  \bZ_2 \cong V^\natural$. With this, the claim is equivalent to the assertion that $(V^\natural)\sslash \bZ_{2p} \cong V^\natural$ for $p=3,5,7,13$. But $V^\natural$, hence all of its subalgebras, is strongly local in the sense of \cite{CKLW}, and so the vertex-algebraic isomorphism $(V^\natural)\sslash \bZ_{2p} \cong V^\natural$ transfers to the analytic setting.
\end{proof}

The primes $2$, $3$, $5$, $7$, and $13$ are precisely those primes $p$ for which $p-1$ divides $24$.  We will use Proposition~\ref{ALYtheorem} to set up various finite group T-dualities.  Let $g \in \bM$ be a generator of the $\bZ_p$-action determined by the orbifold construction $V^\natural = V_\Lambda \sslash  \bZ_p$ from Proposition~\ref{ALYtheorem}. Then $g$ is of $\bM$-conjugacy class $p\rB$, denoted $p-$ in \cite{MR554399}. The image of the dual generator $\widehat{g} \in \widehat{\Lambda}.\Co_0$ in $\Co_1$ is in $\Co_1$-conjugacy class $p\rA$ when $p\neq 2$, and when $p=2$ the image of $\widehat{g}$ in $\Co_0$ is the central element. T-duality explains the oft-observed coincidence between centralizers of elements $p\rB \in \bM$ and centralizers of elements $p\rA \in \Co_1$ (compare \cite{MR554399}):

\begin{corollary} \label{orbifold corollary}
 Let $g \in \bM$ denote an element of $\bM$-conjugacy class $p\rB$ and let $\widehat{g} \in \widehat{\Lambda}.\Co_0$ denote the T-dual element. The centralizers $C(g) \subset \bM$ and $C(\widehat{g}) \subset \widehat{\Lambda}.\Co_0$ are T-dual.
 These centralizers are:
 $$\begin{array}{c||c|c}
 p & C(g) \subset \bM & C(\widehat{g}) \subset \widehat{\Lambda}.\Co_0 \\ \hline 
 2 & 2^{1+24}.\Co_1 & 2^{24}.\Co_0 \\
 3 & 3^{1+12}.2\Suz & 3^{12}:6\Suz \\
 5 & 5^{1+6}:2\rJ_2 & 5^6:(5 \times 2\rJ_2) \\
 7 & 7^{1+4} : 2A_7 & 7^4:(7 \times 2A_7) \\
 13 & 13^{1+2} : 2A_4 & 13^2:(13\times 2A_4)\\
 \end{array}$$
As in the ATLAS \cite{ATLAS}, $\Suz$ denotes the Suzuki group and $\rJ_2 = \mathrm{HJ}$ denotes the Hall--Janko group.  \qed
\end{corollary}

\begin{lemma} \label{moonshine lemma 5 7 13}
  The order of $\omega^\natural$ is not divisible by $5$, $7$, or $13$.
\end{lemma}

\begin{proof} 
Throughout this proof, we let $p=5$, $7$, or $13$ and $d = 24/(p-1)$.
For $p=5$ and $13$, we work with the T-dual pairs $G = C(g) \subset \bM$ and $\widehat{G} = C(\widehat{g}) \subset \widehat{\Lambda}.\Co_0$ from Corollary~\ref{orbifold corollary}. For $p=7$, we need a pair of slightly larger groups. The normalizer $N(7\rB)$ of an element of $\bM$-conjugacy class $7\rB$ has shape $7^{1+4}:(3 \times 2S_7)$; we will use the subgroup $G = 7^{1+4}:(3 \times 2A_7)$ and its dual $\widehat{G} = 7^4:(7:3 \times 2A_7)$. The factor of $3$ acts on $7^{4}$ by scalars and rescales the symplectic pairing. In all cases, we set $G = p.J$ and $\widehat{G} = \widehat{p}.J$, and we have $J = p^d : 2X$ with:
$$\begin{array}{c|c}
p & X \\ \hline
5 & \rJ_2 \\
7 & 3 \times A_7 \\
13 & A_4 \\
\end{array}$$
In all cases $G = p^{1+d}:X$ contains the $p$-Sylow of $\bM$ and so the $p$-part of the order of $\omega^\natural$ is detected by the order of $\omega^\natural|_G$.

Following the notation from Proposition~\ref{t duality prop}, let $\kappa \in \H^2(J,\bZ_p)$ classify the extension $G = p.J$ and expand $\omega^\natural|_G = \alpha + \beta$ with $\alpha \in \H^2(J,\bZ_p)$ classifying the extension $\widehat{G} = \widehat{p}.J$. 
In all cases, the extension $\widehat{G} = \widehat{p}:J$ splits, and so $\alpha = 0$. It follows that $\omega^\natural|_{G} = \beta$ is pulled back from a class $\beta \in \H^3(J,\rU(1))$ and $\widehat{\omega} = \kappa \oplus \beta \in \H^2(J,p) \oplus \H^3(J,\rU(1)) \subset \H^3(\widehat{G},\rU(1))$. By construction,
$\widehat{\omega}$ is the anomaly of the action of $\widehat{G}$ on $V_\Lambda$, and $\beta = \widehat{\omega} |_J$.
We will compute $\H^\bullet(J,\rU(1))$ by using the LHS spectral sequence $\H^\bullet(2X,\H^\bullet(p^d,\rU(1))) \Rightarrow \H^\bullet(J,\rU(1))$.

The K\"unneth formula gives $\H^1(p^d,\rU(1)) \cong p^d$, $\H^2(p^d,\rU(1)) \cong \Alt^2(p^d)$, and $\H^3(p^d,\rU(1))\cong \Sym^2(p^d).\Alt^3(p^d)$. The central $2 \subset 2X$ acts nontrivially on $\H^1(p^d,\rU(1))$ and on $\Alt^3(p^d)$, and so $2X$ can have no cohomology with these coefficients. The module $p^d$ carries a symplectic form but is not symmetrically self-dual over $2X$, and so $\Sym^2(p^d)$ also has no fixed points; thus $\H^0(2X,\H^3(p^d,\rU(1))) = 0$. 

We claim furthermore that $\H^1(2X,\Alt^2(p^d)) = 0$. 
The case $p=13$ is automatic since $13$ does not divide $|2A_4|$. For $p=7$, note that the central $3 \subset 2X$ acts nontrivially on $\Alt^2(7^4)$, and so $\H^i(3 \times 2A_7, 7^5) = 0$ for all~$i$. For $p=5$, one can check $\H^1(2\rJ_2,\Alt^2(5^6)) = 0$ using Holt's GAP program Cohomolo.

It follows that the restriction $\H^3(J,\rU(1)) \to \H^3(2X,\rU(1))$ is an isomorphism, and so it suffices to compute the order of $\widehat{\omega}|_{2X}$, where again $\widehat{\omega}$ is the anomaly for the Leech lattice conformal field theory $V_\Lambda$.
But the action of $2X$ on $V_\Lambda$ factors through a group of shape $2^{24}.\Co_0$ --- indeed, the centralizer in $\Aut(V_\Lambda) = \widehat{\Lambda}.\Co_0$ of the central $2 \subset 2X$ has this shape --- and so $\widehat{\omega}|_{2X}$ is the restriction of a class in $\H^3(2^{24}.\Co_0,\rU(1))$. By Theorem~\ref{conway theorem}, $\H^3(2^{24}.\Co_0,\rU(1))_{(p)} = 0$ for $p\geq 5$.
\end{proof}

\begin{lemma} \label{omega natural order 3}
  $\omega^\natural$ has order $3 \times 2^k$ for some $k$.
\end{lemma}

\begin{proof}
Set $p = 3$, $d = 12$, $J = 3^{12}:2\Suz$, $G = 3^{1+12}.2\Suz$, and $\widehat{G} = 3^{12}:6\Suz$.
It suffices to show that the component of $\omega^\natural|_G$ in $\H^3(G,\rU(1))_{(3)}$ has order exactly $3$.
Unlike in the $p\geq 5$ case, the extension $\widehat{G} = 3.J$ does not split, and so $\alpha \neq 0 \in \H^2(J,\bZ_3)$. It follows that $\omega^\natural|_G$ is non-zero and has order divisible by $3$.

We study the cohomology of the group $G = 3^{1+12}.2\Suz$ using a LHS spectral sequence. Set $E = p^d = 3^{12}$ and $3E = 3^{1+12}$. Then $E$ is symplectically but not symmetrically self-dual over $2\Suz$. We have $\H^1(3E,\rU(1)) \cong E$ and $\H^2(3E,\rU(1)) \cong \Alt^2(E) / 3 = 3^{64}.3$. It follows that $\H^\bullet(2\Suz,\H^1(3E,\rU(1))) = 0$, and the spectral sequence has $E_2$ page: 
$$ \begin{array}{cccc}
 \H^0(2\Suz,\H^3(3E,\rU(1))) & \\
 * & \H^1(2\Suz,3^{64}.3) &  \\
 0 & 0 & 0 & 0  \\
 * & * & * & \H^3(2\Suz,\rU(1))
\end{array} $$
Furthermore, $3E$ is T-dual to a group of shape $E \times 3$. It follows that 
$\omega^\natural|_{3E}$ is pulled back from $\H^3(E,\rU(1)) \cong \Sym^2(E) \oplus \Alt^3(E)$. As in Lemma~\ref{moonshine lemma 5 7 13}, $\H^0(2\Suz,\Sym^2(E) \oplus \Alt^3(E)) = 0$. Therefore $\omega^\natural|_G$ lives in an extension of subquotients of $\H^1(2\Suz,3^{64}.3)$ and $\H^3(2\Suz,\rU(1))$.

Since $\H^1(2\Suz,3^{64}.3)$ has exponent $3$, to complete the proof it suffices to show that $\H^3(2\Suz,\rU(1))_{(3)} = \H^3(\Suz,\rU(1))_{(3)} = 0$.
According to the ATLAS \cite{ATLAS}, the $3$-Sylow in $\Suz$ is contained in a maximal subgroup of shape $3^{5}:M_{11}$. The group $M_{11}$ has two 5-dimensional modules over $\bF_3$, dual to each other. Matching the ATLAS's letters, we will call these modules $3^{5a}$ and $3^{5b}$. Holt's GAP program Cohomolo quickly computes
\begin{gather*}
  \H^1(M_{11},3^{5a}) = 0, \\ \H^1(M_{11},3^{5b}) = 3.
\end{gather*}
But $\H^2(\Suz,\rU(1))_{(3)} = 3$, since the Schur multiplier of $\Suz$ is $\bZ_6$. Since $\H^2(M_{11},\rU(1))_{(3)} = 0$, and $\H^2(\Suz,\rU(1))_{(3)} \subset \H^2(3^5:M_{11},\rU(1))$, we must have $\H^1(M_{11},\H^1(3^5,\rU(1))) = 3$, which is possibly only when $\H^1(3^5,\rU(1)) = 3^{5b}$, which forces the maximal subgroup of $\Suz$ to be $3^{5a}:M_{11}$.

Cohomolo also confirms that $$\H^2(M_{11},\H^1(3^{5a},\rU(1))) = \H^2(M_{11},3^{5b}) = 0.$$
Furthermore, since $3^{5a}$ is not self-dual and does not admit an invariant $3$-form, $\H^0(M_{11},\H^3(3^{5a},\rU(1))) = 0$. 
One can quickly we compute the $10$-dimensional $M_{11}$-module $\H^2(3^{5a},\rU(1))$ and check that $$\H^1(M_{11},\H^2(3^{5a},\rU(1))) = 0.$$
Finally, that $\H^3(M_{11},\rU(1))_{(3)} = 0$ is classically known, and quick to compute with Ellis's GAP program HAP.
This confirms that that $\H^3(2\Suz2,\rU(1))_{(3)} = 0$. 
 \end{proof}

\begin{lemma} \label{omega 3c}
  The restriction $\omega^\natural|_{\langle3\rC\rangle}$ to a cyclic subgroup of $\bM$ generated by an element conjugacy class $3\rC$ is nonzero.
\end{lemma}

\begin{proof}
  The group $G = 3^{1+12}.2\Suz \subset \bM$ in the proof of Lemma~\ref{omega natural order 3} has a subgroup of shape $6\Suz$ centralizing any order-2 lift of the central $2 \subset 2\Suz$. It is T-dual to a subgroup of shape $6\Suz \subset \widehat{G} \subset \Aut(V_\Lambda)$, and so $\omega^\natural|_{6\Suz}$ has nontrivial $3$-part.  The central $2 \subset 6\Suz$ fuses in $\bM$ to conjugacy class $2\rB$, and so   
  the restriction $\omega^\natural|_{6\Suz}$ factors through the centralizer $C(2\rB) = 2^{1+24}.\Co_1$.  By \cite[Lemma 3.1]{JFT}, the nonzero elements of $\H^3(2^{1+24}.\Co_1)_{(3)} = \H^3(\Co_1)_{(3)}$ have nontrivial restriction to conjugacy class $3\rD \in \Co_1$, hence to its lift $3\rD \in 2^{1+24}.\Co_1)$, which fuses to $3\rC \in \bM$. (These conjugacy class fusions are contained in the Character Table libraries in GAP.)
\end{proof}

Lemma~\ref{omega 3c} can also verified using conformal field theoretic methods: letting $g \in \bM$ have conjugacy class $3\rC$, the $g$-twined character $\tr(q^{L_0 - 1}g; V^\natural)$ is known to have a nontrivial multiplier (see e.g.\ \cite{MO69222}), and so $\omega^\natural|_{\langle g \rangle}$ is nonzero by the discussion preceding Example~\ref{example permutation}.

\subsection{The prime \emph{p} = 2} \label{monster 2}

To complete the proof of Theorem~\ref{moonshine theorem}, we must study the 2-part of $\omega^\natural$: we wish to show that the order of $\omega^\natural$ is $8 \times(\text{odd})$. The 2-Sylow in $\bM$ is contained in a maximal subgroup of shape $G = 2^{1+24}.\Co_1$, the centralizer of an element of $\bM$ of conjugacy class $2\rB$. As mentioned already in Proposition~\ref{ALYtheorem}, $G$ is T-dual to the centralizer in $\Aut(V_\Lambda) = \widehat{\Lambda}. \Co_0$ of any order-$2$ lift of the central element element of $\Co_0$. This dual group has shape $\widehat{G} = 2^{24}.\Co_0$. It is known that the extension $\widehat{G}$ does not split~\cite{Ivanov09}.

\begin{lemma}\label{even part 1}
  $8\omega^\natural|_G$ is pulled back from $\Co_1$.
\end{lemma}

\begin{proof}
  The LHS spectral sequence for $G = 2^{1+24}.\Co_1$ asserts that $\H^3(G,\rU(1))$ is an extension of subquotients of:
  \begin{enumerate}
    \item $\H^0(\Co_1,\H^3(2^{1+24},\rU(1)))$, and $\H^3(2^{1+24},\rU(1))$ has exponent $4$;
    \item $\H^1(\Co_1,\H^2(2^{1+24},\rU(1)))$, and $\H^2(2^{1+24},\rU(1)) = 2^{274}.2$ has exponent $2$;
    \item $\H^2(\Co_1,\H^1(2^{1+24},\rU(1)))$, and $\H^1(2^{1+24},\rU(1)) = 2^{24}$ has exponent $2$;
    \item $\H^3(\Co_1,\rU(1))$, which by \cite{JFT} is isomorphic to $\bZ_{12}$.
  \end{enumerate}
  It follows that for any class $\alpha \in \H^3(G,\rU(1))$, $16 \alpha$ is pulled back from $\H^3(\Co_1)$.
  Moreover, the only way for $8\alpha$ to fail to be pulled back from  $\H^3(\Co_1)$ is if $\alpha|_{2^{1+24}}$ has order $4$. But $\omega^\natural|_{2^{1+24}}$ implements a T-duality between $2^{1+24}$ and $2 \times 2^{24}$, implying that $\omega^\natural|_{2^{1+24}}$ has order $2$.
\end{proof}

To complete the proof of Theorem~\ref{moonshine theorem}, we use the same binary dihedral group used in \cite{JFT}. By Proposition~\ref{ALYtheorem}, $\Co_0$ has a subgroup (centralizing $7\rA$) of shape $2A_7$. Inside $A_7$, choose a dihedral group $D_8$ of order $8$ --- there is a unique conjugacy class of such subgroups. Its lift to $2A_7$ is a binary dihedral group $2D_8$ of order $16$. This is a McKay group --- the corresponding Dynkin diagram is $D_6$ --- and its third cohomology is $\H^3(2D_8,\rU(1)) \cong \bZ_{16}$. In \cite{JFT} it is shown that the restriction map $\H^3(\Co_0,\rU(1))_{(2)} \to \H^3(2D_8,\rU(1))$ is an injection onto the even subgroup $\bZ_8 \subset \bZ_{16}$.

This binary dihedral group is also naturally a subgroup of $\bM$, since it lives inside the centralizer $7^{1+4}:2A_7$ of an element of $\bM$-conjugacy class $7\rB$. The central $2 \subset 2D_8$ is a lift of the central element of $\Co_0$.
\begin{lemma}\label{even part 2}
  $\omega^\natural|_{2D_8}$ has order $8$.
\end{lemma}
\begin{proof}
 The above discussion implies that the subgroups $2D_8 \subset \bM$ and $2D_8 \subset \Aut(V_\Lambda)$ are T-dual by the usual orbifold construction $V_\Lambda = V^\natural \sslash  \bZ_2$. The classes in $\H^3(2D_8,\rU(1))$ implementing such a T-duality are those with order $8$.
\end{proof}

Combining Lemmas~\ref{even part 1} and~\ref{even part 2} completes the proof of Theorem~\ref{moonshine theorem}:

\begin{lemma} \label{even part 3}
  $\omega^\natural|_G$ has order $8 \times (\text{odd})$.
\end{lemma}
\begin{proof}
  By Lemma~\ref{even part 1}, $8\omega^\natural|_G$ is pulled back from $\Co_1$. By \cite{JFT}, $\H^3(\Co_1,\rU(1))_{(2)} \cong \bZ_4$, and classes are detected by restricting to $D_8$, and so they are also detected by pulling back to $2D_8$. But $8\omega^\natural|_{2D_8} = 0$ by Lemma~\ref{even part 2}.
\end{proof}

\subsection{$\omega^\natural$ is not a Chern class}\label{section not c2}

The first sentence of Theorem~\ref{moonshine theorem} having been verified, we turn to proving the second sentence, which asserts that $\omega^\natural \neq c_2(V)$ for any complex representation $V$ of $\bM$. The argument in this section was suggested to me by D.\ Treumann. Denote by $R(\bM)$ the complex representation ring of $\bM$. Since $\H^1(\bM,\rU(1)) = 0$, $c_2 : R(\bM) \to \H^3(\bM,\rU(1))$ is linear. Let $2D_8$ denote the binary dihedral group centralizing some chosen element of conjugacy class $7\rB$ used in Section~\ref{monster 2}. We know that $\omega^\natural|_{2D_8}$ has order $8$. Our strategy will be to compute the restriction $c_2|_{2D_8} : R(\bM) \to \H^3(2D_8,\rU(1)) \cong \bZ_{16}$ and show that $\omega^\natural|_{2D_8}$ is not in the image. In fact, we will show that $c_2|_{2D_8}$ is the zero map.

\begin{lemma}\label{merging of conjugacy classes}
 The seven conjugacy classes in $2D_8$ --- the identity, the central element, three of order $4$, and two of order $8$ ---
 merge in $\bM$ to the $\bM$-conjugacy classes $1$, $2\rB$, $4\rD$, and~$8\rF$.
\end{lemma}
\begin{proof}
 Choose any element of order $4$ in $2D_8$, and multiply it with the fixed element of conjugacy class $7\rB$ centralized by $2D_8$. One produces an element of order $28$ in $\bM$ whose fourth power is of class $7\rB$ and whose seventh power is the chosen element of order $4$. The unique fourth root of class $7\rB$ is class $28\rA$, and $(28\rA)^7 = 4\rD$, so all order-4 classes in $2D_8$ merge in $\bM$. But $4\rD$ has a unique square root --- $8\rF$ --- so the two classes of order $8$ must also merge. Finally, $(4\rD)^2 = 2\rB$.
\end{proof}

The group $2D_8$ is the McKay group corresponding to the Dynkin diagram $D_6$. Let $V_0$ denote the trivial representation, $V_1$ the one-dimensional representation in which the kernel of the $2D_8$ action is cyclic of order $8$, $V_2$ and $V_3$ the other two one-dimensional representations, $V_4$ the two-dimensional real representation of $D_8$, and $V_5$ and $V_6$ the two faithful irreps. Taking $V_6$ to be the ``defining'' 2-dimensional representation, the McKay graph is:
$$
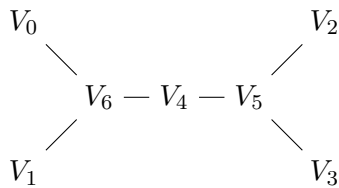
\begin{tikzpicture}
  \path (0,0) node (SW) {$V_1$}
        (1,1) node (A) {$V_6$}
        (0,2) node (NW) {$V_0$}
        (2,1) node (M) {$V_4$}
        (3,1) node (B) {$V_5$}
        (4,2) node (NE) {$V_2$}
        (4,0) node (SE) {$V_3$};
  \draw (A) -- (SW); \draw (A) -- (NW); \draw (A) -- (M);
  \draw (B) -- (SE); \draw (B) -- (NE); \draw (B) -- (M);
\end{tikzpicture}
$$

Given a representation $V$ of $2D_8$, let $n_i$ be defined by $V \cong \bigoplus_{i=0}^6 V_i^{n_i}$. The basic theorem of characters of a finite group says that the numbers $n_i$ can be computed from knowing the traces over $V$ of all conjugacy classes.
Suppose that $V$ is restricted from an irrep of $\bM$. Then, by Lemma~\ref{merging of conjugacy classes}, $\tr_V(g)$ depends only on the order of $g\in 2D_8$.  A quick computation with the character table for $2D_8$ reveals:
\begin{multline*} \bigl(
n_0 , n_1 , n_2 , n_3 , n_4 , n_5 , n_6
\bigr)
= \\
\bigl(\tr_V(1) , \tr_V(2\rB) , \tr_V(4\rD) , \tr_V(8\rF)
\bigr) \cdot
\begin{pmatrix}
 1/16&  1/16&  1/16&  1/16&   1/8&   1/8&   1/8\\
 1/16&  1/16&  1/16&  1/16&   1/8&  -1/8&  -1/8\\
  5/8&  -3/8&   1/8&   1/8&  -1/4&     0&     0\\
  1/4&   1/4&  -1/4&  -1/4&     0&     0&     0
\end{pmatrix}\end{multline*}
Note that $n_2 = n_3$ and $n_5 = n_6$. Using the character table for $\bM$ from the Atlas, we can now compute the map $R(\bM) \to R(2D_8)$. Somewhat surprisingly:

\begin{lemma} \label{multiples of 8}
  With notation as above, if $V$ is a $\bM$-representation then $n_4$ is divisible by $8$ and $n_5=n_6$ is divisible by $16$. The numbers $n_0$, $n_1$, and $n_2=n_3$ can all be odd.\qed
\end{lemma}

The cohomology of McKay groups is easy to understand. It is easiest to work with $\bZ$-coefficients, since $\H^\bullet(G,\bZ) = \H^{\bullet-1}(G,\rU(1))$ is a graded ring. If $G$ is McKay, this ring is supported only in even degrees. As holds for any finite group, $\H^2(G,\bZ) \cong \hom(G,\rU(1))$ consists of first Chern classes of one-dimensional representations. If $G$ is McKay, then from the long exact sequence associated to $\SU(2)/G \to BG \to B\SU(2)$ one can show that $\H^4(G,\bZ) \cong \bZ_{|G|}$. It has a distinguished generator: $c_2$ of the two-dimensional defining representation $G \to \SU(2)$.

We wish to understand the values of $c_2(V)$ for $V$ a representation of $G = 2D_8$. The following argument is based on \cite[\S6]{JFT}. We first  give ourselves a variable $t$ of degree $2$ to track degrees of mixed-degree expressions, so that the total Chern character of a representation $V$ is $c(V) = 1 + c_1(V)t + c_2(V)t^2 + \dots$. The Whitney sum formula $c(V\oplus W) = c(V) c(W)$ says that to calculate $c_2(V)$ for $V$ arbitrary, it suffices to know the values of $c_1(V_i)$ and $c_2(V_i)$ for $V_i$ an irrep. For $G = 2D_8$, we have $\H^2(2D_8,\bZ) = \bZ_2^2$. The three non-zero elements are $c_1(V_1)$, $c_1(V_2)$, and $c_1(V_3)$ --- $c_1(V_1) = c_1(V_2) + c_1(V_3)$. Since $|2D_8| = 16$, $\H^4(2D_8,\bZ) = \bZ_{16}$, with distinguished generator $c_2(V_6)$. We will work with this generator throughout, so that for example ``$c_2(V) = 3$'' means $c_2(V) = 3 c_2(V_6)$. 

We claim:
$$ c_1(V_1)^2 = 0.$$
Indeed, $c_1(V_1)^2$ is certainly 2-torsion, but $V_1$ is pulled back from a one-dimensional representation of $2D_{16}$, and the restriction map $\bZ_{32} \cong \H^4(2D_{16},\bZ) \to \H^4(2D_8,\bZ) \cong \bZ_{16}$ necessarily vanishes on the 2-torsion subgroup of the domain.

By working in $\SU(2)$ we know that $c_1(\Sym^2(V_6)) = 0$ and $c_2(\Sym^2(V_6)) = 4$. Over $2D_8$, $\Sym^2(V_6) = V_1 \oplus V_4$, and so
\begin{multline*} 1 + 4t^2 + \cdots = (1 + tc_1(V_1))(1 + tc_1(V_4) + t^2 c_2(V_4)) \\ = 1 + 2tc_1(V_1) + t^2(c_1(V_1)^2 + c_2(V_4)) + \cdots = 1 + t^2c_2(V_4) + \cdots \end{multline*}
since $c_1(V_1)$ is $2$-torsion and $ c_1(V_1)^2 = 0$. Thus
$$ c_2(V_4) = 4.$$

Continuing on, we have $\Sym^3(V_6) = V_5 \oplus V_6$, and from $\mathrm{SU}(2)$ we learn that $c_2(\Sym^3(V_6)) = 10$, so
$$ c_2(V_5) = 9.$$
Finally, $\Sym^4(V_6) = V_1 \oplus V_4 \oplus (V_2 \oplus V_3)$ and $c_2(\Sym^4(V_6)) = 20$, and so
$$ c_2(V_2 \oplus V_3) = 16 = 0 \mod 16.$$

Given Lemma~\ref{multiples of 8} and the fact that $c_1(V_1)^2 = 0$, we conclude:

\begin{lemma}\label{restrictions of c2s vanish}
  For every $\bM$-representation $V$, $c_2(V)|_{2D_8} = 0 \in \H^3(2D_8,\rU(1))$. \qed
\end{lemma}
Since $\omega^\natural|_{2D_8}$ has order $8$, it cannot be a Chern class. The relation $c_2 = -2 \frac{p_1}2$ verifies that $\omega^\natural$ also is not a fractional Pontryagin class. This completes the proof of Theorem~\ref{moonshine theorem}.

Lemma~\ref{restrictions of c2s vanish} is our main evidence for Conjecture~\ref{c2=0 conjecture}, which predicts that $c_2(V) = 0$ already in $\H^3(\bM,\rU(1))$. We end by observing that Conjecture~\ref{c2=0 conjecture} follows from Conjecture~\ref{H4M conjecture}:

\begin{lemma}\label{conjecture implies conjecture}
 If $\H^3(\bM,\rU(1)) \cong \bZ_{24}$, then $c_2(V) = 0$ for every representation $V$ of~$\bM$.
\end{lemma}
\begin{proof}
  If $\H^3(\bM,\rU(1)) \cong \bZ_{24}$, then restricting to $2D_8$ gives an injection $\H^3(\bM,\rU(1))_{(2)} \to \H^3(2D_8,\rU(1))$. The 2-part of the Lemma then follows from Lemma~\ref{restrictions of c2s vanish}. 
  The 3-part of the Lemma follows from Lemma~\ref{omega 3c}, since a character table computation confirms that $c_2(V) |_{\langle 3\rC\rangle} = 0$ for every $V$, and if $\H^3(\bM,\rU(1)) \cong \bZ_{24}$ then $\H^3(\bM,\rU(1))_{(3)} \to \H^3(\langle 3\rC\rangle,\rU(1))$ is an isomorphism.
\end{proof}

\subsection{Further remarks on the value of $\omega^\natural$} \label{closing remarks}

We calculated the order of the Moonshine anomaly~$\omega^\natural$ without needing to say much about its ambient cohomology group $\H^3(\bM,\rU(1))$. In particular, it is possible that $\H^3(\bM,\rU(1))$ has elements of order much higher than $24$, and it is an interesting question to find out how close $\omega^\natural$ comes to saturating $\H^3(\bM,\rU(1))$. The calculations in this section came from trying (and failing) to answer that question, and in particular to prove Conjecture~\ref{H4M conjecture} that $\omega^\natural$ generates $\H^3(\bM,\rU(1))$.

Since $\omega^\natural|_{2^{1+24}}$ implements a T-duality between $2^{1+24}$ and $2 \times 2^{24}$, it pulls back from a class $\omega_0 \in \H^3(2^{24},\rU(1))$: specifically, $\omega_0$ is the anomaly of the $2^{24}$-action on $V_\Lambda = V^\natural \sslash  \bZ_2$.  The lattice conformal field theory $V_\Lambda$ is constructed so that the anomaly for $\rU(1)^{24} \cong \widehat{\Lambda} \subset \Aut(V_\Lambda)$ is precisely the Leech lattice pairing thought of as a class in $\H^4(\rB\rU(1)^{24},\bZ) =  \Sym^2(\bZ^{24})$. It follows that $\omega_0$ the mod-2 reduction of the Leech pairing.

Although T-duality implies that $\omega^\natural|_{2^{1+24}}$ is pulled back from $2^{24}$, T-duality also guarantees that $\omega^\natural|_{2^{1+24}.\Co_1} $ is \emph{not} pulled back from $2^{24}.\Co_1$, since $2^{24}.\Co_0 \not\cong 2 \times (2^{24}.\Co_1)$. It is interesting to compare the LHS spectral sequences for $2^{1+24}.\Co_1$ and $2^{24}.\Co_1$. We have
\begin{gather*} 
\H^1(2^{1+24},\rU(1)) = \H^1(2^{24},\rU(1)) = 2^{24}, \\
\H^2(2^{1+24},\rU(1)) = 2^{274}.2, \quad \H^2(2^{24},\rU(1)) = 2.2^{274}.2.
\end{gather*}
The equality in the first line emphasizes that the pullback map $\H^1(2^{24},\rU(1)) \to \H^1(2^{1+24},\rU(1))$ is an isomorphism; the pullback $\H^2(2^{24},\rU(1)) \to \H^2(2^{1+24},\rU(1))$ is the obvious surjection.

Unfortunately, the group $\Co_1$ and the modules $2^{274}.2$ and $2.2^{274}.2$ are much too large for programs like Cohomolo and HAP to handle. Fortunately, given a $G$-module $M$, $\H^1(G,M)$ is ``just a matrix computation'' once one has a presentation of $G$. The ATLAS \cite{ATLAS} does not claim a presentation for $\Co_1$, but two presentations were found in \cite{MR886429}, and explicit matrices for one of these presentations were found by R.\ Parker via the methods of \cite{MR2228643}; this presentation, and the generators satisfying it, are listed in Appendix~\ref{Soicher matrices}. Using this presentation, it is  long but not difficult to compute:
$$ \H^1(\Co_1, 2^{274}.2) \cong \H^1(\Co_1,2.2^{274}.2) \cong \bZ_2.$$

Consider now the following long exact sequence produced by taking $\Co_1$ cohomology of the extension $2 \to 2.2^{274}.2 \to 2^{274}.2$:
$$\begin{tikzpicture}[anchor=base, yscale=.75]
 \path 
 (0,1) node {$2$} (2,1) node {$2.2^{274}.2$} (4,1) node {$2^{274}.2$}
 (-2,0) node {$\H^0(\Co_1,-)$} (-2,-1) node {$\H^1(\Co_1,-)$} (-2,-2) node {$\H^2(\Co_1,-)$}
 (0,0) node (a0) {$1$} (2,0) node (b0) {$1$} (4,0) node (c0) {$0$}
 (0,-1) node (a1) {$0$} (2,-1) node (b1) {$1$} (4,-1) node (c1) {$1$}
 (0,-2) node (a2) {$1$} (2,-2) node (b2) {$?$} 
 ;
 \draw[->] (a0) -- node[auto] {$\scriptstyle \sim$} (b0);
 \draw[->] (b1) -- node[auto] {$\scriptstyle \sim$} (c1);
 \draw[right hook->] (a2) -- (b2);
\end{tikzpicture}$$
Here numbers denote the dimensions over $\bF_2$ of the cohomology groups, and we do not draw in arrows that vanish. 

One finds in particular that the pullback map $\H^1(\Co_1,2.2^{274}.2) \to \H^1(\Co_1,2^{274}.2)$ is an isomorphism. In total degree $3$, therefore, the $E_2$ pages of the LHS spectral sequences for $\H^\bullet(2^{1+24}.\Co_1,\rU(1))$ and $\H^\bullet(2^{24}.\Co_1,\rU(1))$ agree except for the $\H^0(\Co_1,\H^3(\dots))$ entry, where the class we want is $\omega_0$, in the image of the pullback map:
$$
\begin{array}{cccc}
 \omega_0 & \\
 0 & \bZ_2 & \H^2(2^{274}.2) \\
 0 & 0 & \H^2(2^{24}) & \H^3(2^{24})  \\
 * & 0 & \bZ_2 & \bZ_{12} \\[12pt] 
 \multicolumn{4}{c}{\H^\bullet(\Co_1,\H^\bullet(2^{1+24},\rU(1)))}
\end{array} 
\hspace{1in} 
\begin{array}{cccc}
 \omega_0 & \\
 \bZ_2 & \bZ_2 & \H^2(2.2^{274}.2) \\
 0 & 0 & \H^2(2^{24}) & \H^3(2^{24}) \\
 * & 0 & \bZ_2 & \bZ_{12} \\[12pt] 
 \multicolumn{4}{c}{\H^\bullet(\Co_1,\H^\bullet(2^{24},\rU(1)))}
\end{array} 
$$
To save space we have abbreviated $\H^i(\Co_1,M)$ by $\H^i(M)$.

How is it possible, then, for $\omega^\natural|_{2^{1+24}.\Co_1}$ not to be pulled back from $2^{24}.\Co_1$? 
The only option is if, in fact, $\omega_0 \in \H^3(2^{24},\rU(1))$ does not extend to a class in $\H^3(2^{24}.\Co_1,\rU(1))$. It fails to extend exactly when the $d_2$ differential, connecting the $E_2$ and $E_3$ pages, is such that $\d_2\omega_0 \in \H^2(2.2^{274}.2)$ is non-zero and in the kernel of $\H^2(2.2^{274}.2) \to \H^2(2^{274}.2)$. By the above long exact sequence analysis, this kernel is the image of $\H^2(\Co_1,2) = \bZ_2$ in $\H^2(2.2^{274}.2)$, i.e.\ it ``is'' the extension $\Co_0 = 2.\Co_1$. 
This is where $\omega^\natural$ ``stores'' the information that the T-dual to $2^{1+24}.\Co_1$ is $2^{24}.\Co_0$ and not $2^{24}.\Co_1 \times 2$.

One way Conjecture~\ref{H4M conjecture} might be proved is if in fact $\H^3(2^{1+24}.\Co_1,\rU(1)) \cong \bZ_{24}$. For this to succeed, the $\bZ_2 = \H^1(\Co_1,2^{274.2})$ in the LHS for $\H^\bullet(2^{1+24}.\Co_1,\rU(1))$ will have to emit a nontrivial differential to either $\H^3(\Co_1,2^{24})$ or to $\H^4(\Co_1,\rU(1))$. These groups seem far beyond what current computer power can handle. 

Furthermore, one will need to handle the group  $\H^2(\Co_1,2^{24})$. 
Direct computation of $\H^2(\Co_1,2^{24})$ requires not just a presentation of $\Co_1$ but also a complete list of syzygies, which is currently inaccessible. Certainly $\H^2(\Co_1,2^{24}) \neq 0$: it contains the class $\kappa$ classifying the extension $2^{24}.\Co_1$, which is known not to split \cite{Ivanov09}. It is reasonable to speculate that this is the only possible extension, so that $\H^2(\Co_1,2^{24}) \cong \bZ_2$. In the notation of Section~\ref{t-duality}, the differential $\d_2 : \H^2(\Co_1,2^{24}) \to \H^4(\Co_1,\rU(1))$ is $\langle -\cup\kappa\rangle$ where $\langle,\rangle$ is the Leech pairing. A possible step towards proving Conjecture~\ref{H4M conjecture} would be to show that $\langle \kappa \cup \kappa \rangle \neq 0$. Using the arguments of Section~\ref{monster 2}, such a computation would already show that $\H^3(2^{1+24}.\Co_1,\rU(1))_{(2)}$, and hence $\H^3(\bM,\rU(1))_{(2)}$, has exponent~$8$.

\appendix

\section{Soicher's presentation for $\Co_1$}\label{Soicher matrices}

The presentation of $\Co_1$ from \cite{MR886429} is as a quotient of the Coxeter group with diagram
$$

$$
by the further relations
$$ (cd)^4 = a, \quad (bcde)^8 = 1, \quad ((b(cd)^2efgh)^{13}i)^3 = 1. $$
Listed below are a collection of $24\times 24$ matrices over $\bF_2$, found by R.~Parker using the methods of~\cite{MR2228643}, that satisfy these relations:

\begin{center}
\hspace*{-1in}
\tiny 
%
\hspace*{-1in} 
\end{center}

\normalsize


\begin{thebibliography}{DMNO13}

\bibitem[ALY17]{AbeLamYamada}
Toshiyuki Abe, Ching~Hung Lam, and Hiromichi Yamada.
\newblock A remark on {${\mathbb Z}_p$}-orbifold constructions of the moonshine
  vertex operator algebra.
\newblock 2017.
\newblock \arXiv{1705.09022}.

\bibitem[Arl84]{MR753419}
Dominique Arlettaz.
\newblock Chern-{K}lassen von ganzzahligen und rationalen {D}arstellungen
  diskreter {G}ruppen.
\newblock {\em Math. Z.}, 187(1):49--60, 1984.
\newblock \DOI{10.1007/BF01163165}. \MRnumber{753419}.

\bibitem[Bal12]{MR3093932}
Benjamin Balsam.
\newblock {\em Turaev-{V}iro theory as an extended {TQFT}}.
\newblock ProQuest LLC, Ann Arbor, MI, 2012.
\newblock Thesis (Ph.D.)--State University of New York at Stony Brook.

\bibitem[BDH09]{BDH0}
Arthur Bartels, Christopher~L. Douglas, and Andr{\'e} Henriques.
\newblock Conformal nets and local field theory.
\newblock 2009.
\newblock \arXiv{0912.5307}.

\bibitem[BDH13]{BDH3}
Arthur Bartels, Christopher~L. Douglas, and Andr{\'e} Henriques.
\newblock Conformal nets {III}: Fusion of defects.
\newblock 2013.
\newblock \arXiv{1310.8263}.

\bibitem[BDH15]{MR3439097}
Arthur Bartels, Christopher~L. Douglas, and Andr{\'e} Henriques.
\newblock Conformal nets {I}: Coordinate-free nets.
\newblock {\em Int. Math. Res. Not. IMRN}, (13):4975--5052, 2015.
\newblock \DOI{10.1093/imrn/rnu080}. \MRnumber{3439097}.

\bibitem[BDH16]{BDH4}
Arthur Bartels, Christopher~L. Douglas, and Andr{\'e} Henriques.
\newblock Conformal nets {IV}: The 3-category.
\newblock 2016.
\newblock \arXiv{1605.00662}.

\bibitem[BDH17]{MR3656522}
Arthur Bartels, Christopher~L. Douglas, and Andr{\'e} Henriques.
\newblock Conformal nets {II}: Conformal blocks.
\newblock {\em Comm. Math. Phys.}, 354(1):393--458, 2017.
\newblock \DOI{10.1007/s00220-016-2814-5}. \MRnumber{3656522}.

\bibitem[B{\'e}n67]{MR0220789}
Jean B{\'e}nabou.
\newblock Introduction to bicategories.
\newblock In {\em Reports of the {M}idwest {C}ategory {S}eminar}, pages 1--77.
  Springer, Berlin, 1967.
\newblock \MRnumber{0220789}.

\bibitem[BGK16]{BGK2016}
Lakshya Bhardwaj, Davide Gaiotto, and Anton Kapustin.
\newblock State sum constructions of spin-tfts and string net constructions of
  fermionic phases of matter.
\newblock 2016.
\newblock \arXiv{1605.01640}.

\bibitem[Bis12]{BischoffThesis}
Marcel Bischoff.
\newblock {\em Construction of Models in low-dimensional Quantum Field Theory
  using Operator Algebraic Methods}.
\newblock PhD thesis, Universit\`a degli Studi di Roma Tor Vergata, 2012.
\newblock \url{https://math.vanderbilt.edu/bischom/thesis/phd.pdf}.

\bibitem[BKL15]{MR3424476}
Marcel Bischoff, Yasuyuki Kawahigashi, and Roberto Longo.
\newblock Characterization of 2{D} rational local conformal nets and its
  boundary conditions: the maximal case.
\newblock {\em Doc. Math.}, 20:1137--1184, 2015.
\newblock \MRnumber{3424476}. \arXiv{1410.8848}.

\bibitem[BT17]{BT2017}
Lakshya Bhardwaj and Yuji Tachikawa.
\newblock On finite symmetries and their gauging in two dimensions.
\newblock 2017.
\newblock \arXiv{1704.02330}.

\bibitem[CCN{\etalchar{+}}85]{ATLAS}
J.~H. Conway, R.~T. Curtis, S.~P. Norton, R.~A. Parker, and R.~A. Wilson.
\newblock {\em Atlas of finite groups}.
\newblock Oxford University Press, Eynsham, 1985.
\newblock Maximal subgroups and ordinary characters for simple groups, With
  computational assistance from J. G. Thackray.

\bibitem[CdLW16]{CLW2016}
Miranda Cheng, Paul de~Lange, and Daniel Whalen.
\newblock Generalised umbral moonshine.
\newblock 2016.
\newblock \arXiv{1608.07835}.

\bibitem[CKLW15]{CKLW}
Sebastiano Carpi, Yasuyuki Kawahigashi, Roberto Longo, and Mih{\'a}ly Weiner.
\newblock From vertex operator algebras to conformal nets and back.
\newblock 2015.
\newblock \arXiv{1503.01260}.

\bibitem[CM16]{CarnahanMiyamoto}
Scott Carnahan and Masahiko Miyamoto.
\newblock Regularity of fixed-point vertex operator subalgebras.
\newblock 2016.
\newblock \arXiv{1603.05645}.

\bibitem[CN79]{MR554399}
J.~H. Conway and S.~P. Norton.
\newblock Monstrous moonshine.
\newblock {\em Bull. London Math. Soc.}, 11(3):308--339, 1979.
\newblock \DOI{10.1112/blms/11.3.308}. \MRnumber{554399}.

\bibitem[Dav13]{Davydov2013}
A.~Davydov.
\newblock Bogomolov multiplier, double class-preserving automorphisms and
  modular invariants for orbifolds.
\newblock 2013.
\newblock \arXiv{1312.7466}.

\bibitem[DG12]{MR2928458}
Chongying Dong and Robert~L. Griess, Jr.
\newblock Integral forms in vertex operator algebras which are invariant under
  finite groups.
\newblock {\em J. Algebra}, 365:184--198, 2012.
\newblock \DOI{10.1016/j.jalgebra.2012.05.006}. \MRnumber{2928458}.
  \arXiv{1201.3411}.

\bibitem[DGH98]{MR1618135}
Chongying Dong, Robert~L. Griess, Jr., and Gerald H{\"o}hn.
\newblock Framed vertex operator algebras, codes and the {M}oonshine module.
\newblock {\em Comm. Math. Phys.}, 193(2):407--448, 1998.
\newblock \DOI{10.1007/s002200050335}. \MRnumber{1618135}.

\bibitem[DH]{Tring}
Christopher~L. Douglas and Andr{\'e} Henriques.
\newblock Geometric string structures.
\newblock \url{http://andreghenriques.com/PDF/TringWP.pdf}.

\bibitem[DH11]{MR2742433}
Christopher~L. Douglas and Andr{\'e}~G. Henriques.
\newblock Topological modular forms and conformal nets.
\newblock In {\em Mathematical foundations of quantum field theory and
  perturbative string theory}, volume~83 of {\em Proc. Sympos. Pure Math.},
  pages 341--354. Amer. Math. Soc., Providence, RI, 2011.
\newblock \DOI{10.1090/pspum/083/2742433}. \arXiv{1103.4187}.

\bibitem[DM94]{MR1302011}
Chongying Dong and Geoffrey Mason.
\newblock The construction of the moonshine module as a {$Z_p$}-orbifold.
\newblock In {\em Mathematical aspects of conformal and topological field
  theories and quantum groups ({S}outh {H}adley, {MA}, 1992)}, volume 175 of
  {\em Contemp. Math.}, pages 37--52. Amer. Math. Soc., Providence, RI, 1994.
\newblock \DOI{10.1090/conm/175/01836}. \MRnumber{1302011}.

\bibitem[DMC15]{MR3376736}
John F.~R. Duncan and Sander Mack-Crane.
\newblock The moonshine module for {C}onway's group.
\newblock {\em Forum Math. Sigma}, 3:e10, 52, 2015.
\newblock \DOI{10.1017/fms.2015.7}. \MRnumber{3376736}. \arXiv{1409.3829}.

\bibitem[DMNO13]{MR3039775}
Alexei Davydov, Michael M{\"u}ger, Dmitri Nikshych, and Victor Ostrik.
\newblock The {W}itt group of non-degenerate braided fusion categories.
\newblock {\em J. Reine Angew. Math.}, 677:135--177, 2013.
\newblock \MRnumber{3039775}. \arXiv{1009.2117}.

\bibitem[DN99]{MR1745258}
Chongying Dong and Kiyokazu Nagatomo.
\newblock Automorphism groups and twisted modules for lattice vertex operator
  algebras.
\newblock In {\em Recent developments in quantum affine algebras and related
  topics ({R}aleigh, {NC}, 1998)}, volume 248 of {\em Contemp. Math.}, pages
  117--133. Amer. Math. Soc., Providence, RI, 1999.
\newblock \arXiv{math/9808088}. \DOI{10.1090/conm/248/03821}.
  \MRnumber{1745258}.

\bibitem[DPR90]{MR1128130}
R.~Dijkgraaf, V.~Pasquier, and P.~Roche.
\newblock Quasi {H}opf algebras, group cohomology and orbifold models.
\newblock {\em Nuclear Phys. B Proc. Suppl.}, 18B:60--72 (1991), 1990.
\newblock Recent advances in field theory (Annecy-le-Vieux, 1990).

\bibitem[Dun07]{MR2352133}
John~F. Duncan.
\newblock Super-moonshine for {C}onway's largest sporadic group.
\newblock {\em Duke Math. J.}, 139(2):255--315, 2007.
\newblock \DOI{10.1215/S0012-7094-07-13922-X}. \MRnumber{2352133}.
  \arXiv{math/0502267}.

\bibitem[DVVV89]{MR1003430}
Robbert Dijkgraaf, Cumrun Vafa, Erik Verlinde, and Herman Verlinde.
\newblock The operator algebra of orbifold models.
\newblock {\em Comm. Math. Phys.}, 123(3):485--526, 1989.
\newblock \MRnumber{1003430}.

\bibitem[DW90]{MR1048699}
Robbert Dijkgraaf and Edward Witten.
\newblock Topological gauge theories and group cohomology.
\newblock {\em Comm. Math. Phys.}, 129(2):393--429, 1990.
\newblock \MRnumber{1048699}.

\bibitem[EGNO15]{EGNO}
Pavel Etingof, Shlomo Gelaki, Dmitri Nikshych, and Victor Ostrik.
\newblock {\em Tensor categories}, volume 205 of {\em Mathematical Surveys and
  Monographs}.
\newblock American Mathematical Society, Providence, RI, 2015.
\newblock \url{http://www-math.mit.edu/~etingof/egnobookfinal.pdf}.
  MRnumber{3242743}. \DOI{10.1090/surv/205}.

\bibitem[ENO10]{MR2677836}
Pavel Etingof, Dmitri Nikshych, and Victor Ostrik.
\newblock Fusion categories and homotopy theory.
\newblock {\em Quantum Topol.}, 1(3):209--273, 2010.
\newblock With an appendix by Ehud Meir. \DOI{10.4171/QT/6}.
  \MRnumber{2677836}. \arXiv{0909.3140}.

\bibitem[ENO11]{MR2735754}
Pavel Etingof, Dmitri Nikshych, and Victor Ostrik.
\newblock Weakly group-theoretical and solvable fusion categories.
\newblock {\em Adv. Math.}, 226(1):176--205, 2011.
\newblock \DOI{10.1016/j.aim.2010.06.009}. \MRnumber{2735754}.

\bibitem[FFRS06]{MR2259258}
Jens Fjelstad, J\"{u}rgen Fuchs, Ingo Runkel, and Christoph Schweigert.
\newblock T{FT} construction of {RCFT} correlators. {V}. {P}roof of modular
  invariance and factorisation.
\newblock {\em Theory Appl. Categ.}, 16:No. 16, 342--433, 2006.
\newblock \MRnumber{2259258}. \arXiv{hep-th/0503194}.

\bibitem[FLM88]{MR996026}
Igor Frenkel, James Lepowsky, and Arne Meurman.
\newblock {\em Vertex operator algebras and the {M}onster}, volume 134 of {\em
  Pure and Applied Mathematics}.
\newblock Academic Press, Inc., Boston, MA, 1988.
\newblock \MRnumber{996026}.

\bibitem[FQ93]{MR1240583}
Daniel~S. Freed and Frank Quinn.
\newblock Chern-{S}imons theory with finite gauge group.
\newblock {\em Comm. Math. Phys.}, 156(3):435--472, 1993.
\newblock arXiv{hep-th/9111004}. \MRnumber{1240583}.

\bibitem[FRS02]{MR1940282}
J\"{u}rgen Fuchs, Ingo Runkel, and Christoph Schweigert.
\newblock T{FT} construction of {RCFT} correlators. {I}. {P}artition functions.
\newblock {\em Nuclear Phys. B}, 646(3):353--497, 2002.
\newblock \DOI{10.1016/S0550-3213(02)00744-7}. \MRnumber{1940282}.
  \arXiv{hep-th/0204148}.

\bibitem[FRS04a]{MR2026879}
J\"{u}rgen Fuchs, Ingo Runkel, and Christoph Schweigert.
\newblock T{FT} construction of {RCFT} correlators. {II}. {U}noriented world
  sheets.
\newblock {\em Nuclear Phys. B}, 678(3):511--637, 2004.
\newblock \DOI{10.1016/j.nuclphysb.2003.11.026}. \MRnumber{2026879}.
  \arXiv{hep-th/0306164}.

\bibitem[FRS04b]{MR2076134}
J\"{u}rgen Fuchs, Ingo Runkel, and Christoph Schweigert.
\newblock T{FT} construction of {RCFT} correlators. {III}. {S}imple currents.
\newblock {\em Nuclear Phys. B}, 694(3):277--353, 2004.
\newblock \DOI{10.1016/j.nuclphysb.2004.05.014}. \MRnumber{2076134}.
  \arXiv{hep-th/0403157}.

\bibitem[FRS05]{MR2137114}
J\"{u}rgen Fuchs, Ingo Runkel, and Christoph Schweigert.
\newblock T{FT} construction of {RCFT} correlators. {IV}. {S}tructure constants
  and correlation functions.
\newblock {\em Nuclear Phys. B}, 715(3):539--638, 2005.
\newblock \DOI{10.1016/j.nuclphysb.2005.03.018}. \MRnumber{2137114}.
  \arXiv{hep-th/0412290}.

\bibitem[Gan09]{MR2500561}
Nora Ganter.
\newblock Hecke operators in equivariant elliptic cohomology and generalized
  {M}oonshine.
\newblock In {\em Groups and symmetries}, volume~47 of {\em CRM Proc. Lecture
  Notes}, pages 173--209. Amer. Math. Soc., Providence, RI, 2009.
\newblock \MRnumber{2500561}. \arXiv{0706.2898}.

\bibitem[Gan16]{MR3539377}
Terry Gannon.
\newblock Much ado about {M}athieu.
\newblock {\em Adv. Math.}, 301:322--358, 2016.
\newblock \DOI{10.1016/j.aim.2016.06.014}. \MRnumber{3539377}.
  \arXiv{1211.5531}.

\bibitem[GJF17]{GJFa}
Davide Gaiotto and Theo Johnson-Freyd.
\newblock Symmetry protected topological phases and generalized cohomology.
\newblock 2017.
\newblock \arXiv{1712.07950}.

\bibitem[GL96]{MR1410566}
Daniele Guido and Roberto Longo.
\newblock The conformal spin and statistics theorem.
\newblock {\em Comm. Math. Phys.}, 181(1):11--35, 1996.
\newblock \MRnumber{1410566}. \arXiv{hep-th/9505059}.

\bibitem[GPRV13]{MR3108775}
Matthias~R. Gaberdiel, Daniel Persson, Henrik Ronellenfitsch, and Roberto
  Volpato.
\newblock Generalized {M}athieu {M}oonshine.
\newblock {\em Commun. Number Theory Phys.}, 7(1):145--223, 2013.
\newblock \DOI{10.4310/CNTP.2013.v7.n1.a5}. \arXiv{1211.7074}.

\bibitem[GW14]{GuWen}
Zheng-Cheng Gu and Xiao-Gang Wen.
\newblock Symmetry-protected topological orders for interacting fermions:
  Fermionic topological nonlinear $\sigma$ models and a special group
  supercohomology theory.
\newblock {\em Phys. Rev. B}, 90(115141), 2014.
\newblock \DOI{10.1103/PhysRevB.90.115141}. \arXiv{1201.2648}.

\bibitem[HC11]{MO69222}
Andre Henriques and Scott Carnahan.
\newblock {$H^4$ of the Monster}.
\newblock MathOverflow, 2011.
\newblock \url{https://mathoverflow.net/q/69222/}.

\bibitem[HO06]{MR2228643}
Derek~F. Holt and E.~A. O'Brien.
\newblock A computer-assisted analysis of some matrix groups.
\newblock {\em J. Algebra}, 300(1):199--212, 2006.
\newblock \DOI{10.1016/j.jalgebra.2006.02.019}. \MRnumber{2228643}.

\bibitem[Iva09]{Ivanov09}
A.~A. Ivanov.
\newblock {\em The {M}onster group and {M}ajorana involutions}, volume 176 of
  {\em Cambridge Tracts in Mathematics}.
\newblock Cambridge University Press, Cambridge, 2009.

\bibitem[JFT18]{JFT}
Theo Johnson-Freyd and David Treumann.
\newblock {$\mathrm{H}^4(\mathrm{Co}_0;\mathbf{Z}) = \mathbf{Z}/24$}.
\newblock {\em Int. Math. Res. Not. IMRN}, 2018.
\newblock \arXiv{1707.07587}.

\bibitem[JL09]{MR2511638}
David Jordan and Eric Larson.
\newblock On the classification of certain fusion categories.
\newblock {\em J. Noncommut. Geom.}, 3(3):481--499, 2009.
\newblock \DOI{10.4171/JNCG/44}. \MRnumber{2511638}. \arXiv{0812.1603}.

\bibitem[Kir02]{MR1923177}
Alexander Kirillov, Jr.
\newblock Modular categories and orbifold models.
\newblock {\em Comm. Math. Phys.}, 229(2):309--335, 2002.
\newblock \DOI{10.1007/s002200200650}. \MRnumber{1923177}.
  \arXiv{math/0104242}.

\bibitem[KL06]{MR2263720}
Yasuyuki Kawahigashi and Roberto Longo.
\newblock Local conformal nets arising from framed vertex operator algebras.
\newblock {\em Adv. Math.}, 206(2):729--751, 2006.
\newblock \DOI{10.1016/j.aim.2005.11.003}.
  \MRnumber{10.1016/j.aim.2005.11.003}. \arXiv{math/0407263}.

\bibitem[KO02]{MR1936496}
Alexander Kirillov, Jr. and Viktor Ostrik.
\newblock On a {$q$}-analogue of the {M}c{K}ay correspondence and the {ADE}
  classification of {$\mathfrak{sl}_2$} conformal field theories.
\newblock {\em Adv. Math.}, 171(2):183--227, 2002.
\newblock \DOI{10.1006/aima.2002.2072}. \MRnumber{1936496}.
  \arXiv{math/0101219}.

\bibitem[Lep85]{MR820716}
J.~Lepowsky.
\newblock Calculus of twisted vertex operators.
\newblock {\em Proc. Nat. Acad. Sci. U.S.A.}, 82(24):8295--8299, 1985.
\newblock \DOI{10.1073/pnas.82.24.8295}, \MRnumber{820716}.

\bibitem[LR95]{MR1332979}
R.~Longo and K.-H. Rehren.
\newblock Nets of subfactors.
\newblock {\em Rev. Math. Phys.}, 7(4):567--597, 1995.
\newblock Workshop on Algebraic Quantum Field Theory and Jones Theory (Berlin,
  1994). \DOI{10.1142/S0129055X95000232}. \MRnumber{1332979}.
  \arXiv{funct-an/9604008}.

\bibitem[Lur09]{Lur09}
Jacob Lurie.
\newblock On the classification of topological field theories.
\newblock In {\em Current developments in mathematics, 2008}, pages 129--280.
  Int. Press, Somerville, MA, 2009.
\newblock \MRnumber{2555928}. \arXiv{0905.0465}.

\bibitem[LX04]{MR2100058}
Roberto Longo and Feng Xu.
\newblock Topological sectors and a dichotomy in conformal field theory.
\newblock {\em Comm. Math. Phys.}, 251(2):321--364, 2004.
\newblock \DOI{10.1007/s00220-004-1063-1}. \MRnumber{2100058}.
  \arXiv{math/0309366}.

\bibitem[M{\"{o}}l16]{MollerThesis}
Sven M{\"{o}}ller.
\newblock {\em A Cyclic Orbifold Theory for Holomorphic Vertex Operator
  Algebras and Applications}.
\newblock PhD thesis, Technische Universit\"{a}t Darmstadt, 2016.
\newblock \arXiv{1611.09843}.

\bibitem[M{\"u}g10]{MR2730815}
Michael M{\"u}ger.
\newblock On superselection theory of quantum fields in low dimensions.
\newblock In {\em X{VI}th {I}nternational {C}ongress on {M}athematical
  {P}hysics}, pages 496--503. World Sci. Publ., Hackensack, NJ, 2010.
\newblock \DOI{10.1142/9789814304634_0041}. \MRnumber{0909.2537}.

\bibitem[Nik13]{MR3077244}
Dmitri Nikshych.
\newblock Morita equivalence methods in classification of fusion categories.
\newblock In {\em Hopf algebras and tensor categories}, volume 585 of {\em
  Contemp. Math.}, pages 289--325. Amer. Math. Soc., Providence, RI, 2013.
\newblock \DOI{10.1090/conm/585/11607}. \MRnumber{3077244}. \arXiv{1208.0840}.

\bibitem[Soi87]{MR886429}
Leonard~H. Soicher.
\newblock Presentations for {C}onway's group {${\rm Co}_1$}.
\newblock {\em Math. Proc. Cambridge Philos. Soc.}, 102(1):1--3, 1987.
\newblock \DOI{10.1017/S0305004100066986}. \MRnumber{886429}.

\bibitem[Tho86]{MR878978}
C.~B. Thomas.
\newblock {\em Characteristic classes and the cohomology of finite groups},
  volume~9 of {\em Cambridge Studies in Advanced Mathematics}.
\newblock Cambridge University Press, Cambridge, 1986.
\newblock \MRnumber{878978}.

\bibitem[Tho10]{MR2681787}
C.~B. Thomas.
\newblock Moonshine and group cohomology.
\newblock In {\em Moonshine: the first quarter century and beyond}, volume 372
  of {\em London Math. Soc. Lecture Note Ser.}, pages 358--377. Cambridge Univ.
  Press, Cambridge, 2010.
\newblock \MRnumber{2681787}.

\bibitem[Tui92]{MR1165184}
Michael~P. Tuite.
\newblock Monstrous {M}oonshine from orbifolds.
\newblock {\em Comm. Math. Phys.}, 146(2):277--309, 1992.
\newblock \MRnumber{1165184}.

\bibitem[Tur10]{MR2674592}
Vladimir Turaev.
\newblock {\em Homotopy quantum field theory}, volume~10 of {\em EMS Tracts in
  Mathematics}.
\newblock European Mathematical Society (EMS), Z{\"u}rich, 2010.
\newblock Appendix 5 by Michael M{\"u}ger and Appendices 6 and 7 by Alexis
  Virelizier. \DOI{10.4171/086}. \MRnumber{2674592}.

\bibitem[TV10]{TuraevVirelizier}
Vladimir Turaev and Alexis Virelizier.
\newblock On two approaches to 3-dimensional tqfts.
\newblock 2010.
\newblock \arXiv{1006.3501}.

\bibitem[TV17]{MR3674995}
Vladimir Turaev and Alexis Virelizier.
\newblock {\em Monoidal categories and topological field theory}, volume 322 of
  {\em Progress in Mathematics}.
\newblock Birkh{\"{a}}user/Springer, Cham, 2017.
\newblock \DOI{10.1007/978-3-319-49834-8}.

\bibitem[Wen13]{Wen2013}
Xiao-Gang Wen.
\newblock Classifying gauge anomalies through spt orders and classifying
  gravitational anomalies through topological orders.
\newblock {\em Physical Review D}, 88(4), 2013.
\newblock \DOI{10.1103/PhysRevD.88.045013}. \arXiv{1303.1803}.

\bibitem[WG17]{WangGu2017}
Qing-Rui Wang and Zheng-Cheng Gu.
\newblock Towards a complete classification of fermionic symmetry protected
  topological phases in 3d and a general group supercohomology theory.
\newblock 2017.
\newblock \arXiv{1703.10937}.

\bibitem[Xu00]{MR1806798}
Feng Xu.
\newblock Algebraic orbifold conformal field theories.
\newblock {\em Proc. Natl. Acad. Sci. USA}, 97(26):14069--14073, 2000.
\newblock \arXiv{math/0004150}. \DOI{10.1073/pnas.260375597}.
  \MRnumber{1806798}.

\end{thebibliography}

\newcommand{\etalchar}[1]{$^{#1}$}

\end{document}